\documentclass[a4paper,reqno]{amsart}
\usepackage[initials]{amsrefs}
\usepackage{xcolor}
\usepackage{enumitem}
\usepackage{amsmath,amssymb,amsthm,mathrsfs,latexsym}
\usepackage{hyperref}
\hypersetup{colorlinks=true,linkcolor=blue,citecolor=blue,urlcolor=blue}
\evensidemargin=\oddsidemargin
\numberwithin{equation}{section}

\newtheorem{theorem}{Theorem}[section]
\newtheorem{lemma}[theorem]{Lemma}
\newtheorem{proposition}[theorem]{Proposition}
\newtheorem{corollary}[theorem]{Corollary}
\theoremstyle{remark}
\newtheorem{remark}[theorem]{Remark}
\theoremstyle{definition}
\newtheorem{definition}[theorem]{Definition}

\newcommand{\A}{\mathcal{A}}

\newcommand{\Apn}[1]{\A^+_{#1}}
\newcommand{\Aple}[1]{\A^+_{\le #1}}

\newcommand{\zA}{\zeta_{\A}}

\newcommand{\epsq}{\varepsilon_q}
\newcommand{\Gc}{\mathcal{G}}

\newcommand{\disc}{\operatorname{disc}}

\title[The first moment over the even hyperelliptic ensemble]
{The first moment of quadratic Dirichlet $L$-functions \\
in the even hyperelliptic ensemble}

\author{Hwanyup Jung}
\address{Department of Mathematics Education, Chungbuk National University,
Cheongju 28644, Republic of Korea}
\email{hyjung@chungbuk.ac.kr}

\begin{document}
\setlength{\abovedisplayskip}{4pt plus 2pt minus 1pt}
\setlength{\belowdisplayskip}{4pt plus 2pt minus 1pt}
\setlength{\abovedisplayshortskip}{1pt plus 1pt}
\setlength{\belowdisplayshortskip}{2pt plus 1pt minus 1pt}

\begin{abstract}
We establish an asymptotic formula for the first moment of the central values
of quadratic Dirichlet $L$-functions in the even hyperelliptic ensemble
$\mathcal H_{2g+2}$ over $\mathbb F_q[x]$, for every fixed odd prime power $q$.
Working with the completed $L$-function yields an exact two-term central-value
formula with truncation levels $g$ and $g-1$; combining the two truncations
before the square-dual contour shifts leaves the three cubic points
$z^3=q^{-4}$ as the only secondary poles, and they contribute the secondary
term $q^{2g/3+2}\bigl(a_g\,g+b_g\bigr)$ with an error
$O_\varepsilon(q^{g(1+\varepsilon)/2})$. The coefficients are real and
$m\mapsto(a_m,b_m)$ has minimal period exactly three, so the coefficient of
$g$ at the order $q^{2g/3}$ depends on $g$ modulo $3$, in contrast with the
earlier even-degree formulas; the difference is confirmed numerically.
\end{abstract}

\maketitle


\section{Introduction}\label{sec:intro}

Let $q$ be an odd prime power.  We study the first moment of quadratic Dirichlet
$L$-functions at the central point over the even hyperelliptic ensemble,
\begin{equation}\label{eq:moment_intro}
M(g)
:=
\sum_{D\in\mathcal H_{2g+2}}
L\!\left(\tfrac12,\chi_D\right),
\end{equation}
as $g\to\infty$ with $q$ fixed.  Here $\mathcal H_{2g+2}$ is the set of monic square-free
polynomials of degree $2g+2$ in $\mathbb F_q[x]$, and $\chi_D$ is the primitive quadratic
character attached to $D$ (see \S\ref{sec:ff}).

For the odd hyperelliptic ensemble $\mathcal H_{2g+1}$, the first moment was
obtained by Andrade and Keating \cite{AK12}, who also formulated conjectures
for the higher moments \cite{AK14} following the recipe of Conrey, Farmer,
Keating, Rubinstein and Snaith \cite{CFKRS}. Florea \cite{Fl17-1} identified
a secondary term of order $gq^{(2g+1)/3}$ and reduced the error to
$O_\varepsilon(q^{g(1+\varepsilon)/2})$; see \cite{Fl17-2} for the second and
third moments and \cite{Fl17-3} for the fourth. For the even hyperelliptic ensemble $\mathcal H_{2g+2}$,
Jung \cite{Ju20} and Andrade--MacMillan \cite{AM-21} obtained the principal
term and the same error exponent, with lower-order terms of order
$gq^{2g/3}$ whose coefficients do not depend on $g\bmod3$. The results of
\cite{Fl17-1}, \cite{Ju20} and \cite{AM-21} are all stated for prime $q$
with $q\equiv1\bmod4$. For an arbitrary odd prime power $q$, it is shown in
\cite{Ju26} that the complete secondary contribution over
$\mathcal H_{2g+1}$ is $q^{2g/3}\bigl(\alpha_gg+\beta_g\bigr)$, whose
dependence on $g$ has minimal period exactly three.

The central idea of this paper is to work with the \emph{completed
$L$-function}
$$
L^*(s,\chi_D):=\frac{L(s,\chi_D)}{1-q^{-s}}
$$
in place of $L(s,\chi_D)$ itself: for $D\in\mathcal H_{2g+2}$ the character
$\chi_D$ is even, so $L(s,\chi_D)$ carries the trivial factor $1-q^{-s}$,
and $L^*(s,\chi_D)$ has a symmetric functional equation (see \S\ref{sec:ff}). 
We therefore introduce the associated completed first
moment
\begin{equation}\label{eq:completed_moment_intro}
M^*(g)
:=
\sum_{D\in\mathcal H_{2g+2}}
L^*\!\left(\tfrac12,\chi_D\right).
\end{equation}
At the central point one has
$L\!\left(\tfrac12,\chi_D\right)=(1-q^{-1/2})L^*\!\left(\tfrac12,\chi_D\right)$,
and summing this identity over $D\in\mathcal H_{2g+2}$ gives
\begin{equation}\label{eq:recover}
M(g)
=
(1-q^{-1/2})M^*(g).
\end{equation}
The analysis is accordingly carried out for $M^*(g)$, and $M(g)$ is
recovered from \eqref{eq:recover} at the end.

A second departure from the earlier treatments concerns the ground field.
Neither a primality nor a congruence condition is imposed on $q$: throughout
this paper $q$ is an arbitrary odd prime power. The obstruction is that in
general the quadratic Gauss sums over $\mathbb F_q[x]$ carry a root-of-unity
phase, which propagates into the Poisson summation formula built from them.
We therefore work with a phase-normalized Gauss sum, into whose definition
that phase is absorbed, and with the Poisson summation formula it yields;
both are free of the phase for every odd prime power $q$
(see \S\ref{sec:poisson}; cf. \cite[Section~9]{Fl17-1}).

Let $\zA$ denote the zeta function of $\A:=\mathbb F_q[x]$, so that
$\zA(s)=\bigl(1-q^{1-s}\bigr)^{-1}$ (see \S\ref{sec:ff}), and
$$
\mathscr P(s)
:=
\prod_P
\left(
1-\frac{1}{|P|^s(|P|+1)}
\right),
$$
where the product is taken over all monic irreducible polynomials $P\in\A$.

Our main result is the following.

\begin{theorem}\label{thm:main}
Let $q$ be a fixed odd prime power. For every $\varepsilon>0$, as
$g\to\infty$, one has
\begin{equation}\label{eq:main}
M(g)
=
\mathcal M_{\mathrm{prin}}(g)
+
\mathcal S(g)
+
O_\varepsilon\!\left(q^{g(1+\varepsilon)/2}\right),
\end{equation}
where
\begin{equation}\label{eq:Mprin-intro}
\mathcal M_{\mathrm{prin}}(g)
:=
\frac{\mathscr P(1)}{2\zA(2)}
q^{2g+2}
\left[
(2g+2)
+
\frac{4}{\log q}
\frac{\mathscr P'}{\mathscr P}(1)
+
2\zA\!\left(\tfrac12\right)
\right]
\end{equation}
and
\begin{equation}\label{eq:S-intro}
\mathcal S(g)
:=
q^{2g/3+2}
\left(
a_gg+b_g
\right).
\end{equation}
For $m\in\mathbb Z$, the coefficients $a_m$ and $b_m$ are real and are given
by \eqref{eq:ab-def}, and the sequence
$m\mapsto(a_m,b_m)$ has minimal period exactly three.
\end{theorem}

The principal term $\mathcal M_{\mathrm{prin}}(g)$ and the error exponent are
those of \cite[Theorem~1.4]{AM-21} and \cite[Theorem~1.1]{Ju20}. 
The secondary term is again of order $gq^{2g/3}$, but its coefficients, unlike
theirs, depend on $g\bmod3$; Remark~\ref{rem:numerics} shows numerically that
this is the shape realized by $M(g)$.

Passing to the completed $L$-function is what makes it possible to treat the
even hyperelliptic ensemble along the lines followed for the odd hyperelliptic
ensemble in \cite{Fl17-1} and, for the period-three secondary term, in
\cite{Ju26}. It also avoids a computation that the earlier even-degree
treatments do not. In \cite{AM-21} and \cite{Ju20} the principal contribution
is evaluated by
residues at the origin, which brings in the Taylor coefficients of the Euler
product there and, in \cite{AM-21}, an induction on $g$ carried out in an
appendix. Here, as in \cite[Section~8]{Fl17-1}, the corresponding contour
integral is retained rather than evaluated: it reappears with the opposite
sign among the square-dual residues, the two cancel exactly, and no residue
at the origin is computed at all. 
This is a second point at which the present treatment follows the odd-degree
argument of \cite{Fl17-1} rather than the even-degree ones of \cite{AM-21, Ju20}.

The passage to $L^*$ should serve the higher moments as well. Since $L^*$
satisfies the exact two-term central-value formula \eqref{eq:AFE}, the
$k$-th moment of $L^*$ expands into $2^k$ blocks of a single shape, indexed
by truncation levels $N_1,\dots,N_k\in\{g,g-1\}$, and the factor
$(1-q^{-1/2})^k$ returns the $k$-th moment of
$L\!\left(\tfrac12,\chi_D\right)$; the second and third moments over
$\mathcal H_{2g+2}$, obtained for $\mathcal H_{2g+1}$ in \cite{Fl17-2},
should be within reach in this way. With the four-sum identity of
\cite[Lemma~2.1]{Ju20} and \cite[Lemma~2.4]{AM-21} the same moment expands
instead into $4^k$ blocks of two kinds, and the bookkeeping grows quickly
with $k$ (Remark~\ref{rem:four-sums}).

The paper is organized as follows. Section~\ref{sec:setup} develops
the basic function-field notation, the completed $L$-function and
its exact two-term central-value formula, the square-free
character-average identity, the normalized Gauss sums, the Poisson
summation formula, and the decomposition of $M^*(g)$.
Section~\ref{sec:main-term} evaluates the zero-dual contribution.
Section~\ref{sec:square-dual} analyzes the nonzero square dual
variables, proves the principal contour cancellation, evaluates the
secondary residues, and establishes the exact period-three
structure. Section~\ref{sec:nonsquare} proves a uniform fixed-$V$
coefficient estimate and uses it to bound the nonsquare contribution.
Finally, Section~\ref{sec:proof} combines the three contributions and
completes the proof of Theorem~\ref{thm:main}, and closes with a numerical
verification of the period-three secondary term for small $q$ and $g$
(Remark~\ref{rem:numerics}).

\section{Background and setup of the problem}\label{sec:setup}

This section collects the material used below: notation and standard facts
about $\A=\mathbb F_q[x]$ and its quadratic characters, the exact two-term
formula for the central value of the completed $L$-function, the character
average over $\mathcal H_{2g+2}$, the phase-normalized Gauss sums and the
Poisson summation formula, and the decomposition of $M^*(g)$.

\subsection{Basic facts about function fields and \texorpdfstring{$L$}{L}-functions}\label{sec:ff}

Throughout the paper, $q$ denotes an odd prime power and $\A=\mathbb F_q[x]$.
For a nonzero $f\in\A$, we write $d(f):=\deg f$ and define the norm of $f$ by $|f|:=q^{d(f)}$.
We write $\A^+$ for the set of monic polynomials, $\mathcal P$ for
the set of monic irreducible polynomials, and $\mathcal H$ for the
set of monic square-free polynomials in $\A$; in particular
$1\in\mathcal H$.
Elements of $\mathcal P$ are called \emph{prime polynomials}, and the
letter $P$ is reserved for them throughout.
Unless otherwise stated, a sum over $h\mid f$ is taken over the monic
divisors $h\in\A^+$ of $f$.

In the Vinogradov notation $f\ll g$, equivalently $f=O(g)$, subscripts
indicate the parameters on which the implied constant may depend; since $q$
is fixed throughout, it may always depend on $q$ as well, whether or not
this is indicated.

For any subset $\mathcal{U}$ of $\A^{+}$, write $\mathcal{U}_{n} = \{f \in \mathcal{U} : d(f) = n\}$
and $\mathcal{U}_{\le n} = \{f \in \mathcal{U} : d(f) \le n\}$; both are
understood to be empty whenever $n<0$.
The \emph{even hyperelliptic ensemble} of degree $2g+2$ is the
family $\mathcal H_{2g+2}$.

For $\operatorname{Re}(s)>1$, the zeta function of $\A$ is given by
$$
\zeta_{\A}(s):=\sum_{f\in\A^+}\frac{1}{|f|^{s}}
=\prod_{P\in\mathcal P}\left(1-\frac{1}{|P|^{s}}\right)^{-1}
=\bigl(1-q^{1-s}\bigr)^{-1}
$$
(see \cite[Chapter~2]{Ro02}); the last expression continues $\zeta_\A$
meromorphically to $\mathbb C$, and we write $\zeta_\A$ also for the
continuation, so that $\zeta_\A(\tfrac12)=(1-q^{1/2})^{-1}$. In the variable
$u=q^{-s}$ we write $\mathcal Z_{\A}(u):=(1-qu)^{-1}$, so that
$\mathcal Z_{\A}(q^{-s})=\zeta_{\A}(s)$ and
$\sum_{f\in\A^+}u^{d(f)}=\mathcal Z_\A(u)$ for $|u|<q^{-1}$; below,
$\mathcal Z_\A$ is also used as a meromorphic function on
$\mathbb C\setminus\{q^{-1}\}$.

For $a\in\A$ and $P\in\mathcal{P}$, the quadratic residue symbol
$\bigl(\frac{a}{P}\bigr)\in\{0,\pm1\}$ is defined by
$\left(\frac{a}{P}\right)\equiv a^{(|P|-1)/2}\bmod P$,
so that $\bigl(\frac{a}{P}\bigr)=0$ exactly when $P\mid a$. For
nonconstant $f\in\A^+$ with prime factorization $f=\prod_PP^{e_P}$, the
Jacobi symbol is obtained by multiplicative extension,
$$
\left(\frac{a}{f}\right):=\prod_P\left(\frac{a}{P}\right)^{e_P},
$$
with the convention $\bigl(\frac{a}{1}\bigr)=1$; thus
$\bigl(\frac{a}{f}\bigr)=0$ precisely when $(a,f)\ne1$. For fixed
$f\in\A^+$, the map $a\mapsto\bigl(\frac{a}{f}\bigr)$ depends only on
$a$ modulo $f$; we write $\psi_f(a):=\bigl(\frac{a}{f}\bigr)$ for $a\in\A$, a
quadratic Dirichlet character modulo $f$, which need not be primitive.
Letting $f_1$ be the \emph{odd part} of $f$, that is, the unique square-free
divisor of $f$ for which $f/f_1$ is a square, multiplicativity of the Jacobi
symbol gives
\begin{equation}\label{eq:psi-induced}
\psi_f(a)=\left(\frac{a}{f_1}\right)\mathbf 1_{(a,f)=1}
\qquad(a\in\A),
\end{equation}
so that $\psi_f$ is induced by the primitive character $\psi_{f_1}$ modulo
$f_1$; in particular $\psi_f$ is principal if and only if $f$ is a square.

Let $\chi_2:\mathbb F_q^\times\to\{\pm1\}$ be the quadratic character
defined by $\chi_2(c):=c^{(q-1)/2}$, regarded as taking values $\pm1$ in $\mathbb Z$.
Since $q$ is odd, one has $\bigl(\frac{c}{P}\bigr)=\chi_2(c)^{d(P)}$ for
$c\in\mathbb F_q^\times$ (see \cite[Chapter~3]{Ro02}), and therefore
\begin{equation}\label{eq:constant-symbol}
\psi_f(c)=\left(\frac{c}{f}\right)=\chi_2(c)^{d(f)}
\qquad
(c\in\mathbb F_q^\times,\ f\in\A^+).
\end{equation}
As usual, a Dirichlet character $\chi$ modulo $f$ is called \emph{even} if
$\chi(c)=1$ for every $c\in\mathbb F_q^\times$, and \emph{odd} otherwise.
By \eqref{eq:constant-symbol}, $\psi_f$ is even precisely when $d(f)$ is
even.

For coprime $A,B\in\A^+$, the quadratic reciprocity law over $\A$
(see \cite[Chapter~3]{Ro02}) states that
$$
\left(\frac{A}{B}\right)
=
(-1)^{\frac{q-1}{2}d(A)d(B)}
\left(\frac{B}{A}\right).
$$

For $D\in\mathcal H_{2g+2}$, let $\chi_D:=\bigl(\frac{D}{\cdot}\bigr)$ be the
quadratic character attached to $D$, as in \cite{Fl17-1,Ju20,AM-21}. Since
$d(D)=2g+2$ is even, the reciprocity sign above is $1$ whenever $(D,f)=1$,
and both symbols vanish otherwise; hence
\begin{equation}\label{eq:reciprocity-even}
\chi_D(f)=\left(\frac{D}{f}\right)=\left(\frac{f}{D}\right)=\psi_D(f)
\qquad(f\in\A^+),
\end{equation}
for every odd prime power $q$, with no congruence condition such as
$q\equiv1\bmod4$.
Since $D$ is square-free, its odd part is $D$ itself, so that
\eqref{eq:psi-induced} shows $\chi_D$ to be the primitive quadratic
character modulo $D$.

The following counting fact will be used repeatedly. Let $h\in\A^+$ and $m\ge d(h)$. Every $H\in\A_m^+$ is
uniquely written as $H=A+Bh$ with $d(A)<d(h)$ and $B\in\A_{m-d(h)}^+$;
hence each residue class modulo $h$ contains exactly $q^{m-d(h)}$ monic
polynomials of degree $m$, and consequently
\begin{equation}\label{eq:full-period}
\sum_{H\in\A_m^+}\psi_h(H)
=
q^{m-d(h)}\sum_{A\bmod h}\psi_h(A).
\end{equation}

\subsubsection*{$L$-functions and the completed $L$-function}
Let $\chi$ be a Dirichlet character over $\A$. For
$\operatorname{Re}(s)>1$, the Dirichlet $L$-function attached to $\chi$ is
given by the absolutely convergent series and Euler product
$$
L(s,\chi)
:=
\sum_{f\in\A^+}\frac{\chi(f)}{|f|^s}
=
\prod_{P\in\mathcal P}
\left(1-\frac{\chi(P)}{|P|^{s}}\right)^{-1}
$$
(see \cite[Chapter~4]{Ro02}). In the variable $u=q^{-s}$ one has
$|f|^{-s}=u^{d(f)}$, so that, for $|u|<q^{-1}$,
$$
\mathcal L(u,\chi)
:=
L(s,\chi)
=
\sum_{n\ge0}
\Biggl(
\sum_{f\in\A_n^+}\chi(f)
\Biggr)u^n.
$$

For $D\in\mathcal H_{2g+2}$ the character $\chi_D$ is nonprincipal, $D$
being square-free and non-constant, and $\mathcal L(u,\chi_D)$ is a
polynomial in $u$ of degree at most $2g+1$ (see
\cite[Proposition~4.3]{Ro02}; cf. \cite[Section~2]{Ju20}).

Since $d(D)$ is even, $\mathcal L(u,\chi_D)$ has a trivial zero at $u=1$, so
that the \emph{completed $L$-function}
\begin{equation}\label{eq:Lstar_def}
\mathcal L^*(u,\chi_D)
:=
\frac{\mathcal L(u,\chi_D)}{1-u}
\end{equation}
is a polynomial in $u$ of degree exactly $2g$ with
$\mathcal L^*(0,\chi_D)=1$: it is the numerator of the zeta function of the
hyperelliptic curve $y^2=D(x)$, which has genus $g$ over $\mathbb F_q$, and
it satisfies the functional equation
\begin{equation}\label{eq:FE_Lstar}
\mathcal L^*(u,\chi_D)
=
(qu^2)^g
\mathcal L^*\!\left(\frac{1}{qu},\chi_D\right).
\end{equation}

Since $\mathcal L(u,\chi_D)$ is a polynomial in $u$, the function
$s\mapsto\mathcal L(q^{-s},\chi_D)$ is entire and continues $L(s,\chi_D)$
to all of $\mathbb C$; the central value $L(\tfrac12,\chi_D)$ is understood
as $\mathcal L(q^{-1/2},\chi_D)$.
Evaluating \eqref{eq:Lstar_def} at $u=q^{-1/2}$ therefore gives
$L\!\left(\tfrac12,\chi_D\right) = \mathcal L(q^{-1/2},\chi_D) = (1-q^{-1/2}) \mathcal L^*(q^{-1/2},\chi_D)$.
Since $u=q^{-s}$, the function $\mathcal L^*(u,\chi_D)$ is the completed
$L$-function $L^*(s,\chi_D)$ of \S\ref{sec:intro}; in particular
$L^*\!\left(\tfrac12,\chi_D\right)=\mathcal L^*(q^{-1/2},\chi_D)$. Summing
the identity above over $D\in\mathcal H_{2g+2}$ and using the definitions
\eqref{eq:moment_intro} and \eqref{eq:completed_moment_intro} yields the
recovery relation \eqref{eq:recover}.

\subsection{Two-term formula and preliminary lemmas}

We first identify the coefficients of the completed $L$-function
$\mathcal L^*(u,\chi_D)$.

\begin{lemma}
Let $D\in\mathcal H_{2g+2}$. Then one has 
\begin{equation}\label{eq:Lstar-formal-series}
\mathcal L^*(u,\chi_D)
=
\sum_{f\in\A^+}A_D^*(f)u^{d(f)},
\end{equation}
where
\begin{equation}\label{eq:Astar}
A_D^*(f)
:=
\sum_{\substack{h,m\in\A^+\\ hm=f}}
\frac{\chi_D(h)}{|m|}
=
\frac{1}{|f|}
\sum_{h\mid f}|h|\chi_D(h).
\end{equation}
The series in \eqref{eq:Lstar-formal-series} converges absolutely for
$|u|<q^{-1}$, and \eqref{eq:Lstar-formal-series} also holds as an identity
of formal power series in $u$.
\end{lemma}

\begin{proof}
Regrouping the expansion of $\mathcal L(u,\chi)$ in \S\ref{sec:ff} gives
$\mathcal L(u,\chi_D)=\sum_{h\in\A^+}\chi_D(h)u^{d(h)}$, absolutely
convergent for $|u|<q^{-1}$, while
$\sum_{m\in\A^+}u^{d(m)}|m|^{-1}=(1-u)^{-1}$ for $|u|<1$, the coefficient of
$u^n$ on the left being $q^{-n}\#\A_n^+=1$. Multiplying
$\sum_{h\in\A^+}\chi_D(h)u^{d(h)}$ by $\sum_{m\in\A^+}u^{d(m)}|m|^{-1}$ ---
legitimate for $|u|<q^{-1}$, where both converge absolutely --- and
collecting the terms with $hm=f$ gives \eqref{eq:Lstar-formal-series} with the first expression
in \eqref{eq:Astar}; substituting $m=f/h$ turns it into the second. The same
computation is an identity of formal power series.
\end{proof}

Grouping \eqref{eq:Lstar-formal-series} according to degree expresses
$\mathcal L^*(u,\chi_D)$ as a power series in which the coefficient of $u^N$
is $\sum_{f\in\A_N^+}A_D^*(f)$. Since $\mathcal L^*(u,\chi_D)$ is a
polynomial of degree $2g$, these coefficients vanish for $N>2g$ -- although
the individual values $A_D^*(f)$ with $d(f)>2g$ do not -- and
\begin{equation}\label{eq:Lstar-by-degree}
\mathcal L^*(u,\chi_D)
=
\sum_{N=0}^{2g}
\Biggl(
\sum_{f\in\A_N^+}A_D^*(f)
\Biggr)u^N.
\end{equation}

\begin{lemma}[Exact two-term formula]\label{lem:AFE}
For every $D\in\mathcal H_{2g+2}$, one has
\begin{equation}\label{eq:AFE}
\mathcal L^*\!\left(\frac{1}{\sqrt q},\chi_D\right)
=
\sum_{f\in\A^+_{\le g}}
\frac{A_D^*(f)}{\sqrt{|f|}}
+
\sum_{f\in\A^+_{\le g-1}}
\frac{A_D^*(f)}{\sqrt{|f|}}.
\end{equation}
\end{lemma}

\begin{proof}
By \eqref{eq:Lstar-by-degree}, $\mathcal L^*(u,\chi_D)$ is a polynomial in
$u$ of degree $2g$ whose coefficient of $u^N$ is
$\sum_{f\in\A_N^+}A_D^*(f)$ for $0\le N\le2g$, and by \eqref{eq:FE_Lstar}
it satisfies the same functional equation as $\mathcal L(u,\chi_D)$ for
$D\in\mathcal H_{2g+1}$; the argument of \cite[Lemma~3.3]{AK12} therefore
applies verbatim and gives \eqref{eq:AFE}.
\end{proof}

\begin{remark}[Comparison with earlier central-value formulas]
\label{rem:four-sums}
For $\mathcal H_{2g+2}$, the exact central-value identity of
\cite[Lemma~2.1]{Ju20} and \cite[Lemma~2.4]{AM-21} expresses
$L(\tfrac12,\chi_D)$ as four sums: two Dirichlet-polynomial sums
coming from the functional equation, and two correction sums attached to the
trivial factor $1-u$. Dividing that factor out first makes
\eqref{eq:FE_Lstar} symmetric and absorbs its effect into the coefficients
$A_D^*(f)$ of \eqref{eq:Astar}, leaving the two sums of a single shape in
\eqref{eq:AFE}.
\end{remark}

\subsubsection*{The square-free character average}

The next lemma expresses the average over square-free polynomials in terms
of unrestricted polynomial sums. It is \cite[Lemma~2.2]{Ju20}, the even-degree
analogue of \cite[Lemma~2.2]{Fl17-1}; see also \cite{AM-21}. The proof opens
by replacing $\chi_D(h)$ with $\psi_h(D)$ through quadratic reciprocity, a
step taken in \cite{Fl17-1} under the hypothesis $q\equiv1\bmod4$; by
\eqref{eq:reciprocity-even} it is available here for every odd prime power
$q$, $d(D)$ being even, and the rest of the argument is unchanged.

\begin{lemma}[Character average over $\mathcal H_{2g+2}$]\label{lem:sigma}
For $h\in\A^+$, one has
$$
\sum_{D\in\mathcal H_{2g+2}}\chi_D(h)
=\sum_{\substack{C\in\A^+_{\le g+1}\\ C\mid h^\infty}}\sum_{H\in\A^+_{2g+2-2d(C)}}\psi_h(H)
-q\sum_{\substack{C\in\A^+_{\le g}\\ C\mid h^\infty}}\sum_{H\in\A^+_{2g-2d(C)}}\psi_h(H).
$$
Here $C\mid h^\infty$ means that every prime polynomial factor of $C$ divides $h$.
\end{lemma}

\subsubsection*{The function-field Perron formula}

We will repeatedly use the following coefficient-extraction
identities, commonly referred to as the function-field Perron
formula; see \cite[Lemma~2.8]{AM-21} and
\cite[Fact~3.1]{Ju20}.

\begin{lemma}[Function-field Perron formula]
Let $a:\A^+\to\mathbb C$, and suppose that
$\mathcal F(u) := \sum_{f\in\A^+}a(f)u^{d(f)}$
converges absolutely for $|u|<R$, where $R>0$. If
$n\in\mathbb Z_{\ge0}$ and
$0<r<\min\{R,1\}$, then
\begin{align}
\sum_{f\in\A_n^+}a(f)
&=
\frac{1}{2\pi i}
\oint_{|u|=r}
\frac{\mathcal F(u)}{u^{n+1}}\,du,
\label{eq:perron_exact}\\
\sum_{f\in\A_{\le n}^+}a(f)
&=
\frac{1}{2\pi i}
\oint_{|u|=r}
\frac{\mathcal F(u)}
 {(1-u)u^{n+1}}\,du.
\label{eq:perron_partial}
\end{align}
\end{lemma}

\subsection{Poisson summation formula}\label{sec:poisson}

For a general odd prime power $q$, the quadratic Gauss sums over $\A$ carry a
root-of-unity phase, which is where $q\equiv1\bmod4$ enters \cite{Fl17-1}.
Building that phase into the definition of the Gauss sum makes its local
values, its multiplicativity and the Poisson summation formula phase-free
for every odd prime power $q$.

Let $p=\operatorname{char}\mathbb F_q$ and let
$e_p(\alpha):=\exp\bigl(2\pi i\operatorname{Tr}_{\mathbb F_q/\mathbb F_p}(\alpha)/p\bigr)$
be the canonical additive character of $\mathbb F_q$. With $\chi_2$ as in
\S\ref{sec:ff}, put
\begin{equation}\label{eq:g1-def}
g_1:=\sum_{\alpha\in\mathbb F_q^\times}\chi_2(\alpha)e_p(\alpha),
\qquad
\epsq:=\frac{g_1}{\sqrt q}.
\end{equation}
The classical identity $g_1^2=\chi_2(-1)q$
(see \cite[Chapter~10, Section~3]{IR90}) gives
\begin{equation}\label{eq:epsq}
\epsq^2=\chi_2(-1), \qquad \epsq^4=1;
\end{equation}
in particular $|\epsq|=1$, so that $\epsq\overline{\epsq}=1$.
Thus $\epsq\in\{\pm1\}$ if $q\equiv1\bmod4$, and $\epsq\in\{\pm i\}$ if
$q\equiv3\bmod4$; for prime $q$, Gauss's evaluation gives $\epsq=1$ when
$q\equiv1\bmod4$ and $\epsq=i$ when $q\equiv3\bmod4$.

For $a\in\mathbb F_q((x^{-1}))$, write $[x^{-1}]a$ for the coefficient of
$x^{-1}$ in the Laurent expansion of $a$, and put
$e(a):=e_p\bigl([x^{-1}]a\bigr)$. Since $a\mapsto[x^{-1}]a$ is
$\mathbb F_p$-linear and $e_p$ is an additive character, $e$ is a character
of the additive group $\mathbb F_q((x^{-1}))$. Every $a\in\A$ has
$[x^{-1}]a=0$, so $e$ is trivial on $\A$.

For $h\in\A^+$, let $\nu(h)\in\{0,1\}$ be the parity of $d(h)$, so that
$\nu(h)\equiv d(h)\bmod2$.

\begin{definition}[Phase-normalized Gauss sum]
For $h\in\A^+$ and $V\in\A$, define
\begin{equation}\label{eq:Gc-def}
\Gc(V,\psi_h)
:=\overline{\epsq}^{\,\nu(h)}
\sum_{u\bmod h}\psi_h(u)e\!\left(\frac{uV}{h}\right).
\end{equation}
\end{definition}

Since $\psi_h$ is a character modulo $h$ and $e$ is trivial on $\A$, the
summand in \eqref{eq:Gc-def} depends only on $u$ and $V$ modulo $h$; the sum
is therefore independent of the residue system chosen, and $\Gc(V,\psi_h)$
depends only on $V\bmod h$.

\begin{lemma}[Normalized prime Gauss sum]\label{lem:prime-gauss}
For every $P\in\mathcal{P}_{d}$, one has
\begin{equation}\label{eq:prime-gauss}
\Gc(1,\psi_P)=q^{d/2}=\sqrt{|P|}.
\end{equation}
\end{lemma}

\begin{proof}
Let $\theta_1,\dots,\theta_d$ be the roots of $P$, put $\theta:=\theta_1$ and
$K:=\mathbb F_q(\theta)\cong\mathbb F_{q^d}$, so that $f\mapsto f(\theta)$ is a
ring isomorphism $\A/(P)\to K$. The roots of $P$ are distinct and lie in
$K$, so for $u\in\A$ with $d(u)<d$ the partial-fraction decomposition of
$u/P$ over $K$ reads
$$
\frac{u}{P}=\sum_{i=1}^{d}\frac{u(\theta_i)}{P'(\theta_i)}
\cdot\frac{1}{x-\theta_i},
$$
and $(x-\theta_i)^{-1}=\sum_{n\ge0}\theta_i^nx^{-n-1}$ gives
$[x^{-1}](u/P)=\sum_{i=1}^{d}u(\theta_i)/P'(\theta_i)$, the $d$ summands
being the images of $u(\theta)/P'(\theta)$ under the
$\mathbb F_q$-automorphisms of $K$, so that
$$
e\!\left(\frac{u}{P}\right)
=e_p\!\left(\operatorname{Tr}_{K/\mathbb F_q}
\!\left(\frac{u(\theta)}{P'(\theta)}\right)\right).
$$
Let $\chi_K$ be the quadratic character of $K^\times$ (with $\chi_K(0)=0$)
and $g_d:=\sum_{\bar u\in K}\chi_K(\bar u)
e_p\bigl(\operatorname{Tr}_{K/\mathbb F_q}(\bar u)\bigr)$, so that $g_1$ is
the sum in \eqref{eq:g1-def}. Transporting
$\bigl(\frac uP\bigr)\equiv u^{(|P|-1)/2}\bmod P$ along $f\mapsto f(\theta)$
gives $\psi_P(u)=\chi_K(u(\theta))$, and substituting
$\bar u\mapsto\bar uP'(\theta)$, which permutes $K$ since $P'(\theta)\ne0$,
gives $\sum_{u\bmod P}\psi_P(u)e(u/P)=\chi_K(P'(\theta))g_d$.

Since $\chi_K=\chi_2\circ N_{K/\mathbb F_q}$ and
$N_{K/\mathbb F_q}(P'(\theta))=\prod_{i\ne j}(\theta_i-\theta_j)
=(-1)^{d(d-1)/2}\disc P$, the Pellet--Stickelberger theorem
\cite[Corollary~1]{Sw62} ($\chi_2(\disc P)=(-1)^{d-1}$) gives
$\chi_K(P'(\theta))=\chi_2(-1)^{d(d-1)/2}(-1)^{d-1}$, while the
Hasse--Davenport relation \cite[Chapter~11, Section~4]{IR90}
gives $g_d=(-1)^{d-1}g_1^{\,d}$. The signs $(-1)^{d-1}$ cancel, and
$\chi_2(-1)=\epsq^2$ by \eqref{eq:epsq} and $g_1=\epsq\sqrt q$ by
\eqref{eq:g1-def} give
\begin{equation}\label{eq:prime-gauss-unnormalized}
\sum_{u\bmod P}\psi_P(u)e\!\left(\frac{u}{P}\right)
=\epsq^{\,d(d-1)}\epsq^{\,d}q^{d/2}=\epsq^{\,d^2}q^{d/2}.
\end{equation}
Since $d^2\equiv\nu(P)\bmod4$ and $\epsq^4=1$, multiplying
\eqref{eq:prime-gauss-unnormalized} by $\overline{\epsq}^{\,\nu(P)}$ as in
\eqref{eq:Gc-def} gives \eqref{eq:prime-gauss}.
\end{proof}

\begin{lemma}[Phase-free quadratic Gauss sums]\label{lem:Gc-values} 
\begin{enumerate}[label=\textup{(\arabic*)}]
\item
Let $P\in\mathcal{P}_{d}$, and $i\ge1$.
For $V\in\A\setminus\{0\}$, write $V=P^k V_1$ with $k\ge0$ and $P\nmid V_1$. Then one has
\begin{equation}\label{eq:Gc-values}
\Gc(V,\psi_{P^i})
=\begin{cases}
0& i\le k,\ i\text{ odd},\\
\phi(P^i)& i\le k,\ i\text{ even},\\
-|P|^{i-1}& i=k+1,\ i\text{ even},\\
\bigl(\frac{V_1}{P}\bigr)|P|^{i-1}\sqrt{|P|}& i=k+1,\ i\text{ odd},\\
0& i\ge k+2,
\end{cases}
\end{equation}
where $\phi(f)$ denotes the polynomial Euler totient function over $\A$.

\item
For coprime $h_1,h_2\in\A^+$ and $V\in\A$, one has
\begin{equation}\label{eq:Gc-mult}
\Gc(V,\psi_{h_1h_2})=\Gc(V,\psi_{h_1})\Gc(V,\psi_{h_2}).
\end{equation}

\item
For $h\in\A^+$, one has
\begin{equation}\label{eq:Gc-zero}
\Gc(0,\psi_h)=\phi(h)\mathbf 1_{h=\square}.
\end{equation}
\end{enumerate}
\end{lemma}

\begin{proof}
We prove \eqref{eq:Gc-values} first, distinguishing the two parities of $i$.
Suppose first that $i$ is even. Then $\nu(P^i)=0$ and $\psi_{P^i}(u)=\bigl(\frac{u}{P^i}\bigr)=\mathbf 1_{P\nmid u}$.
Thus $\Gc(V,\psi_{P^i})=\sum_{u\bmod P^i,\,P\nmid u}e\bigl(uV/P^i\bigr)$.
Removing from the full
sum over $u\bmod P^i$ the terms with $P\mid u$, which are parameterized by
$u=Pv$ with $v\bmod P^{i-1}$, and using additive orthogonality in the form
$$
\sum_{u\bmod P^j}e\!\left(\frac{uV}{P^j}\right)=|P|^j\mathbf 1_{P^j\mid V}
\qquad(j\ge0),
$$
we obtain $\Gc(V,\psi_{P^i})=|P|^i\mathbf 1_{P^i\mid V}-|P|^{i-1}\mathbf 1_{P^{i-1}\mid V}$.
Since $V=P^kV_1$ with $P\nmid V_1$, the indicator $\mathbf 1_{P^j\mid V}$
equals $1$ for $j\le k$ and $0$ for $j>k$. Applying this to $j=i$ and to
$j=i-1$, and using $\phi(P^i)=|P|^i-|P|^{i-1}$, we obtain
$$
\Gc(V,\psi_{P^i})
=\begin{cases}
\phi(P^i)& i\le k,\\
-|P|^{i-1}& i=k+1,\\
0& i\ge k+2,
\end{cases}
$$
which are the entries of \eqref{eq:Gc-values} with $i$ even.

Suppose next that $i$ is odd. Then $\nu(P^i)=\nu(P)$ and
$\psi_{P^i}(u)=\bigl(\frac uP\bigr)$. Writing $u=a+Pb$ with $a\bmod P$ and
$b\bmod P^{i-1}$, we have $\psi_{P^i}(u)=\bigl(\frac aP\bigr)$ and, since
$V=P^kV_1$, $uV/P^i=aV_1/P^{\,i-k}+bV_1/P^{\,i-k-1}$; hence
\eqref{eq:Gc-def} factors as
$$
\Gc(V,\psi_{P^i})
=\overline{\epsq}^{\,\nu(P)}
\left(\sum_{a\bmod P}\Bigl(\frac aP\Bigr)e\!\left(\frac{aV_1}{P^{\,i-k}}\right)\right)
\left(\sum_{b\bmod P^{i-1}}e\!\left(\frac{bV_1}{P^{\,i-k-1}}\right)\right).
$$
If $i\ge k+2$, the summand of the second sum depends only on
$b\bmod P^{\,i-k-1}$, $e$ being trivial on $\A$, so that sum equals
$|P|^{k}\sum_{c\bmod P^{\,i-k-1}}e\bigl(cV_1/P^{\,i-k-1}\bigr)=0$ by the
orthogonality relation above, since $i-k-1\ge1$ and $P\nmid V_1$.
If $i\le k+1$, then $bV_1/P^{\,i-k-1}\in\A$, so the second sum equals
$|P|^{i-1}$. For $i\le k$ one has $aV_1/P^{\,i-k}\in\A$ as well, and the
first sum is $\sum_{a\bmod P}\bigl(\frac aP\bigr)=0$; for $i=k+1$, replacing
$a$ by $aV_1^{-1}$, where $V_1^{-1}$ is the inverse of $V_1$ modulo $P$, and
using $\bigl(\frac{V_1^{-1}}{P}\bigr)=\bigl(\frac{V_1}{P}\bigr)$, turns the
first sum into
$\bigl(\frac{V_1}{P}\bigr)\epsq^{\,\nu(P)}\Gc(1,\psi_P)$, so that
$\Gc(V,\psi_{P^i})=\bigl(\frac{V_1}{P}\bigr)|P|^{i-1}\sqrt{|P|}$ by
Lemma~\ref{lem:prime-gauss}. This proves \eqref{eq:Gc-values}.

Now let $(h_1,h_2)=1$ and put $d_j=d(h_j)$. Writing
$u=u_1h_2+u_2h_1\bmod{h_1h_2}$, the Chinese remainder theorem and the
quadratic reciprocity law give, as in \cite[Lemma~3.2]{Fl17-1},
\begin{align*}
\sum_{u\bmod h_1h_2}\psi_{h_1h_2}(u)e\!\left(\frac{uV}{h_1h_2}\right)
=\chi_2(-1)^{d_1d_2}
\prod_{j=1}^2\left(\sum_{u_j\bmod h_j}\psi_{h_j}(u_j)e\!\left(\frac{u_jV}{h_j}\right)\right),
\end{align*}
the sign $\chi_2(-1)^{d_1d_2}$ being $-1$ when $q\equiv3\bmod4$ and
$d_1,d_2$ are both odd, so that the unnormalized sums are not multiplicative
in general. For $j=1,2$ one has
$\sum_{u_j\bmod h_j}\psi_{h_j}(u_j)e\bigl(u_jV/h_j\bigr)
=\epsq^{\,\nu(h_j)}\Gc(V,\psi_{h_j})$ by \eqref{eq:Gc-def} and
$|\epsq|=1$, and $\chi_2(-1)=\epsq^2$ by \eqref{eq:epsq}; hence
\begin{align*}
\Gc(V,\psi_{h_1h_2})
&=\epsq^{\,\nu(h_1)+\nu(h_2)-\nu(h_1h_2)+2d_1d_2}
\Gc(V,\psi_{h_1})\Gc(V,\psi_{h_2}).
\end{align*}
Since $\nu(h_1h_2)=\nu(h_1)+\nu(h_2)-2\nu(h_1)\nu(h_2)$, the exponent equals
$2\bigl(\nu(h_1)\nu(h_2)+d_1d_2\bigr)$; and $\nu(h_j)\equiv d_j\bmod2$ gives
$\nu(h_1)\nu(h_2)\equiv d_1d_2\bmod2$, so the exponent is divisible by $4$.
As $\epsq^4=1$ by \eqref{eq:epsq}, this proves \eqref{eq:Gc-mult}.

Finally, by definition,
$\Gc(0,\psi_h)=\overline{\epsq}^{\,\nu(h)}\sum_{u\bmod h}\psi_h(u)$.
The character $\psi_h$ is principal precisely when $h$ is a square.
In that case $\nu(h)=0$, and the character sum equals $\phi(h)$; otherwise it vanishes.
This proves \eqref{eq:Gc-zero}.
\end{proof}

With the Gauss sum $\Gc$ normalized as above, the Poisson summation formula
over $\A$ takes the following form for every odd prime power $q$.

\begin{proposition}[Poisson summation over $\A$]
Let $h\in\A^+$ and $m\ge0$.
If $d(h)$ is even, then
\begin{equation}\label{eq:poisson_even}
\sum_{H\in\Apn{m}}\psi_h(H)
=\frac{q^m}{|h|}\Biggl(\Gc(0,\psi_h)+(q-1)\sum_{V\in\Aple{d(h)-m-2}}\Gc(V,\psi_h)-\sum_{V\in\Apn{d(h)-m-1}}\Gc(V,\psi_h)\Biggr).
\end{equation}
If $d(h)$ is odd, then
\begin{equation}\label{eq:poisson_odd}
\sum_{H\in\Apn{m}}\psi_h(H)
=\frac{q^m\sqrt q}{|h|}\sum_{V\in\Apn{d(h)-m-1}}\Gc(V,\psi_h).
\end{equation}
\end{proposition}

\begin{proof}
First suppose that $m\ge d(h)$. Then $d(h)-m-2<0$ and $d(h)-m-1<0$, so the
$V$-sums in \eqref{eq:poisson_even} and \eqref{eq:poisson_odd} are empty,
while \eqref{eq:full-period} gives
$\sum_{H\in\Apn{m}}\psi_h(H)=q^{m-d(h)}\sum_{A\bmod h}\psi_h(A)$. If
$d(h)$ is even, then $\nu(h)=0$, so \eqref{eq:Gc-def} gives
$\sum_{A\bmod h}\psi_h(A)=\Gc(0,\psi_h)$, and $|h|=q^{d(h)}$ turns the
right-hand side into $q^m\Gc(0,\psi_h)/|h|$: this is
\eqref{eq:poisson_even}. If $d(h)$ is odd, then $h$ is not a square, so
$\Gc(0,\psi_h)=0$ by Lemma~\ref{lem:Gc-values}(3) and both sides of
\eqref{eq:poisson_odd} vanish.

It remains to consider $0\le m\le d(h)-1$. For $u\bmod h$, set
$S_m(u):= \sum_{H\in\Apn{m}}e\!\left(\frac{uH}{h}\right)$. For $B\in\A$, the
map $u\mapsto e(uB/h)$ is a character of $\A/h\A$, trivial precisely when
$h\mid B$, so that
$|h|^{-1}\sum_{u\bmod h}e\bigl(u(A-H)/h\bigr)=\mathbf 1_{A\equiv H\bmod h}$
for $A,H\in\A$. Writing
$\psi_h(H)=\sum_{A\bmod h}\psi_h(A)\mathbf 1_{A\equiv H\bmod h}$ and then
making the changes of variables $u\mapsto-u$ and $A\mapsto-A$, we obtain
\begin{align*}
\sum_{H\in\Apn{m}}\psi_h(H)
&=
\frac{\psi_h(-1)}{|h|}
\sum_{u\bmod h}S_m(u)
\sum_{A\bmod h}\psi_h(A)e\!\left(\frac{uA}{h}\right)
=
\frac{\psi_h(-1)\epsq^{\nu(h)}}{|h|}
\sum_{u\bmod h}S_m(u)\Gc(u,\psi_h),
\end{align*}
the factor $\epsq^{\nu(h)}$ arising because the inner sum is the
\emph{unnormalized} Gauss sum of \eqref{eq:Gc-def} and $\epsq\overline{\epsq}=1$.
By \eqref{eq:constant-symbol} the scalar factor outside the last sum is
$\psi_h(-1)\epsq^{\nu(h)}/|h|=\chi_2(-1)^{d(h)}\epsq^{\nu(h)}/|h|$.

Clearly $S_m(0)=q^m$. Let $u\bmod h$ be nonzero, represented by a polynomial
with $d(u)<d(h)$, and write $u/h=\sum_{n\ge1}c_nx^{-n}$. Writing
$H=x^m+\sum_{j<m}a_jx^j$, so that $H$ runs over $\Apn{m}$ as
$a_0,\dots,a_{m-1}$ run independently over $\mathbb F_q$, and using
$[x^{-1}]\bigl(ua_jx^j/h\bigr)=a_jc_{j+1}$, the additivity of $e$ factors
$S_m(u)$ into
$e(ux^m/h)\prod_{j<m}\sum_{a\in\mathbb F_q}e_p(ac_{j+1})$; each inner sum is
$q$ or $0$ according as $c_{j+1}$ vanishes or not, $e_p$ being nontrivial.
Hence, as in \cite[Proposition~3.1]{Fl17-1}, $S_m(u)=q^me(ux^m/h)$ if
$c_n=0$ for every $1\le n\le m$, and $S_m(u)=0$ otherwise. The first
nonvanishing coefficient of $u/h$ is $c_{d(h)-d(u)}$, so the former holds
precisely when $d(u)\le d(h)-m-1$.

Write each nonzero $u\bmod h$ with $d(u)\le d(h)-m-1$ uniquely as $u=cV$
with $c\in\mathbb F_q^\times$ and $V\in\A^+$, so that
$d(V)=d(u)\le d(h)-m-1$. The substitution $u\mapsto c^{-1}u$ in \eqref{eq:Gc-def} together with
\eqref{eq:constant-symbol} gives
$\Gc(cV,\psi_h)=\chi_2(c)^{d(h)}\Gc(V,\psi_h)$; and, $h$ and $V$ being monic,
$cVx^m/h=cx^{d(V)+m-d(h)}+\cdots$, so that $e(cVx^m/h)$ equals $1$ when
$d(V)\le d(h)-m-2$ and $e_q(c)$ when $d(V)=d(h)-m-1$. Grouping the
nonzero residues accordingly, we obtain
\begin{align}\label{eq:poisson-intermediate}
\sum_{H\in\Apn{m}}\psi_h(H)
=
\frac{q^m\chi_2(-1)^{d(h)}\epsq^{\nu(h)}}{|h|}
\Biggl[
\Gc(0,\psi_h)
+\Sigma_0
\sum_{V\in\Aple{d(h)-m-2}}\Gc(V,\psi_h)
+\Sigma_1
\sum_{V\in\Apn{d(h)-m-1}}\Gc(V,\psi_h)
\Biggr],
\end{align}
where, by \eqref{eq:constant-symbol},
$$
\Sigma_0:=\sum_{c\in\mathbb F_q^\times}\psi_h(c)
=\sum_{c\in\mathbb F_q^\times}\chi_2(c)^{d(h)},
\qquad
\Sigma_1:=\sum_{c\in\mathbb F_q^\times}\psi_h(c)e_q(c)
=\sum_{c\in\mathbb F_q^\times}\chi_2(c)^{d(h)}e_q(c).
$$
It remains to evaluate the scalar prefactor and the two constants $\Sigma_0$
and $\Sigma_1$ in each parity.

Suppose first that $d(h)$ is even. Then $\nu(h)=0$ and
$\chi_2(-1)^{d(h)}=1$, so the prefactor is $q^m/|h|$. Moreover
$\chi_2(c)^{d(h)}=1$ for every $c\in\mathbb F_q^\times$, whence
$\Sigma_0=q-1$ and $\Sigma_1=\sum_{c\in\mathbb F_q^\times}e_q(c)=-1$,
the second because $e_q$ is a nontrivial additive character of
$\mathbb F_q$. Therefore \eqref{eq:poisson-intermediate} reduces to
\eqref{eq:poisson_even}.

Suppose now that $d(h)$ is odd. Then $h$ is not a square, so
$\Gc(0,\psi_h)=0$ by Lemma~\ref{lem:Gc-values}(3) and the first term of
\eqref{eq:poisson-intermediate} is absent. Here $\chi_2(c)^{d(h)}=\chi_2(c)$,
so that $\Sigma_0=\sum_{c\in\mathbb F_q^\times}\chi_2(c)=0$ and
$\Sigma_1=\sum_{c\in\mathbb F_q^\times}\chi_2(c)e_q(c)=g_1=\epsq\sqrt q$,
so that only the term carrying $\Sigma_1$ survives in
\eqref{eq:poisson-intermediate}. Here three powers of $\epsq$ occur: $\psi_h(-1)=\chi_2(-1)=\epsq^2$ by
\eqref{eq:epsq}, the factor $\epsq^{\nu(h)}=\epsq$ left by the
normalization, and $\Sigma_1=\epsq\sqrt q$. The prefactor in
\eqref{eq:poisson-intermediate} is thus $q^m\epsq^3/|h|$, and the
coefficient of $\sum_{V\in\Apn{d(h)-m-1}}\Gc(V,\psi_h)$ is
$q^m\epsq^{3}\Sigma_1/|h|=q^m\epsq^{4}\sqrt q/|h|=q^m\sqrt q/|h|$, the three
exponents summing to $4$. Hence
\eqref{eq:poisson-intermediate} reduces to \eqref{eq:poisson_odd}.
\end{proof}

\begin{remark}
Up to the factor $\overline{\epsq}^{\,\nu(h)}$, $\Gc(V,\psi_h)$ is the
generalized Gauss sum $G(u,\chi)$ of \cite[Section~3]{Fl17-1}, in which the
letters $u$ and $V$ play the opposite roles; when $q\equiv1\bmod4$ is prime
one has $\epsq=1$, the two agree, and \eqref{eq:Gc-values} reduces to the
table of local values found there. For a general odd prime power $q$, the
unnormalized sum $\sum_{u\bmod P}\psi_P(u)e(u/P)$ carries the phase
$\epsq^{\,d^2}$ by \eqref{eq:prime-gauss-unnormalized}, and the factor
$\overline{\epsq}^{\,\nu(h)}$ built into \eqref{eq:Gc-def} cancels that
phase; every value in Lemma~\ref{lem:Gc-values} is free of it, and the three
powers of $\epsq$ in the proof of \eqref{eq:poisson_odd} combine to
$\epsq^4=1$. Under the hypothesis $q\equiv1\bmod4$ of \cite{Fl17-1} each
factor is separately trivial and no such cancellation is visible. In
\cite[Section~9]{Fl17-1} the primality assumption is removed while
$q\equiv1\bmod4$ is retained; there the phase $\epsq\in\{\pm1\}$ survives in
the odd-degree entry of the local table, and the case $\epsq=-1$ is not
carried out.

The exponent must be the parity $\nu(h)$ and not the degree $d(h)$: with
$\overline{\epsq}^{\,d(h)}$ in place of $\overline{\epsq}^{\,\nu(h)}$ in
\eqref{eq:Gc-def} it would be $2d_1d_2$ in the proof of \eqref{eq:Gc-mult},
so that multiplicativity would fail whenever $q\equiv3\bmod4$ and $d_1,d_2$
are odd, and $3d(h)+1$ in the proof of \eqref{eq:poisson_odd}, which is not
divisible by $4$ when $d(h)\equiv3\bmod4$. The same weight is forced in
\eqref{eq:prime-gauss} by $d^2\equiv\nu(P)\bmod4$.
\end{remark}

\subsection{Setup of the moment problem}

The argument follows the first-moment framework of
\cite[Section~3]{AM-21} and \cite[Section~3]{Ju20}, the new feature being
the divisor weight carried by the coefficients $A_D^*(f)$ of
\eqref{eq:Astar}.

For $N\in\{g,g-1\}$, define
\begin{equation}\label{eq:SN_def}
\mathcal S_N^*
:=
\sum_{D\in\mathcal H_{2g+2}}
\sum_{f\in\A_{\le N}^+}
\frac{A_D^*(f)}{\sqrt{|f|}}.
\end{equation}
Substituting \eqref{eq:Astar} into \eqref{eq:SN_def} and interchanging the
finite sums, we obtain
\begin{equation}\label{eq:SN_expanded}
\mathcal S_N^*
=
\sum_{f\in\A_{\le N}^+}
\frac{1}{|f|^{3/2}}
\sum_{h\mid f}|h|
\sum_{D\in\mathcal H_{2g+2}}\chi_D(h).
\end{equation}
Applying the identity in Lemma~\ref{lem:sigma} to the innermost character
average in \eqref{eq:SN_expanded}, we obtain
\begin{align}\label{eq:SN-before-poisson-raw}
\mathcal S_N^*
=
\sum_{f\in\A_{\le N}^+}
\frac{1}{|f|^{3/2}}
\sum_{h\mid f}|h|
\Biggl(\, \sum_{\substack{C\in\A_{\le g+1}^+\\C\mid h^\infty}}
\sum_{H\in\A_{2g-2d(C)+2}^+}\psi_h(H) -q   \sum_{\substack{C\in\A_{\le g}^+\\C\mid h^\infty}}
\sum_{H\in\A_{2g-2d(C)}^+}\psi_h(H) \Biggr).
\end{align}
As in \cite[Section~4]{Fl17-1} (see also \cite[(3.2)]{Ju20} and
\cite[Section~3]{AM-21}), we first discard the terms with $C\in\A_{g+1}^+$
in the first sum of \eqref{eq:SN-before-poisson-raw}. There $2g+2-2d(C)=0$,
so the inner sum over $H$ equals $\psi_h(1)=1$, while
$\#\bigl\{C\in\A_{g+1}^+:C\mid h^\infty\bigr\}\ll_\varepsilon q^{\varepsilon g}$
uniformly for $d(h)\le g$. Writing $f=hr$, these terms therefore contribute
$$
\ll_\varepsilon
q^{\varepsilon g}
\sum_{\substack{h,r\in\A^+\\d(h)+d(r)\le N}}|h|^{-1/2}|r|^{-3/2}
\ll_\varepsilon
q^{\varepsilon g}\sum_{a+b\le N}q^{(a-b)/2}
\ll_\varepsilon
q^{N/2+\varepsilon g}.
$$
Consequently, 
\begin{align}\label{eq:SN-before-poisson}
\mathcal S_N^*
&=
\sum_{f\in\A_{\le N}^+}
\frac{1}{|f|^{3/2}}
\sum_{h\mid f}|h|
\sum_{\substack{C\in\A_{\le g}^+\\C\mid h^\infty}}
\Biggl(
\,\sum_{H\in\A_{2g-2d(C)+2}^+}\psi_h(H)
-q\sum_{H\in\A_{2g-2d(C)}^+}\psi_h(H)
\Biggr)
+O_\varepsilon\!\left(q^{N/2+\varepsilon g}\right).
\end{align}

Let $\mathcal S_{N,\mathrm e}^*$ and $\mathcal S_{N,\mathrm o}^*$ denote the
contributions to the sum in \eqref{eq:SN-before-poisson} from divisors
$h\mid f$ of even and odd degree, respectively. Then
\begin{equation}\label{eq:parity_split}
\mathcal S_N^*
=
\mathcal S_{N,\mathrm e}^*
+
\mathcal S_{N,\mathrm o}^*
+
O_\varepsilon\!\left(q^{N/2+\varepsilon g}\right).
\end{equation}

For $h\in\A^+$ and $C\mid h^\infty$ put
\begin{equation}\label{eq:lambda-def}
\lambda=\lambda(h,C):=d(h)-2g+2d(C),
\end{equation}
and define
\begin{align}
s_{\mathrm e}(h;C)
&:=(q-1)\sum_{V\in\A^+_{\le\lambda-4}}\frac{\Gc(V,\psi_h)}{\sqrt{|h|}}
-\sum_{V\in\A^+_{\lambda-3}}\frac{\Gc(V,\psi_h)}{\sqrt{|h|}}
\nonumber\\
&\qquad
-\frac{q-1}{q}\sum_{V\in\A^+_{\le\lambda-2}}\frac{\Gc(V,\psi_h)}{\sqrt{|h|}}
+\frac{1}{q}\sum_{V\in\A^+_{\lambda-1}}\frac{\Gc(V,\psi_h)}{\sqrt{|h|}},
\label{eq:s-even}\\
s_{\mathrm o}(h;C)
&:=\sum_{V\in\A^+_{\lambda-3}}\frac{\Gc(V,\psi_h)}{\sqrt{|h|}}
-\frac{1}{q}\sum_{V\in\A^+_{\lambda-1}}\frac{\Gc(V,\psi_h)}{\sqrt{|h|}}.
\label{eq:s-odd}
\end{align}

We first consider the even-degree contribution. Apply
\eqref{eq:poisson_even} to the two sums over $H$ in
$\mathcal S_{N,\mathrm e}^*$: the factor $|h|$ in
\eqref{eq:SN-before-poisson} cancels the denominator $|h|$ in
\eqref{eq:poisson_even}, and with $q^{2g+2-2d(C)}=q^{2g+2}/|C|^2$,
$q^{2g+1-2d(C)}=q^{2g+1}/|C|^2$ and $1-q^{-1}=\zA(2)^{-1}$ the terms
carrying a nonzero dual variable $V$ assemble into $s_{\mathrm e}(h;C)$. Writing
$f=hm$ with $m\in\A^+$, so that $d(f)\le N$ becomes $d(h)+d(m)\le N$, we
obtain
\begin{align}\label{eq:SN-even-poisson}
\mathcal S_{N,\mathrm e}^*
&=
\frac{q^{2g+2}}{\zA(2)}
\sum_{\substack{h,m\in\A^+\\ d(h)+d(m)\le N\\ d(h)\text{ even}}}
\frac{\Gc(0,\psi_h)}{|h|^{3/2}|m|^{3/2}}
\sum_{\substack{C\in\A_{\le g}^+\\C\mid h^\infty}}\frac{1}{|C|^2}
+q^{2g+2}
\sum_{\substack{h,m\in\A^+\\ d(h)+d(m)\le N\\ d(h)\text{ even}}}
\frac{1}{|h|\,|m|^{3/2}}
\sum_{\substack{C\in\A_{\le g}^+\\C\mid h^\infty}}
\frac{s_{\mathrm e}(h;C)}{|C|^2}.
\end{align}
Split
$s_{\mathrm e}(h;C)=s_{\mathrm e}(h;C)_{\square}+s_{\mathrm e}(h;C)_{\ne\square}$
according as the dual variable $V$ is a nonzero square or a nonsquare.
Since
$d(h)$ is even, the two exact-degree sums in \eqref{eq:s-even} have odd
degree and so contain no squares, so that $s_{\mathrm e}(h;C)_{\square}$
comes from the two cumulative sums alone. Accordingly
\begin{equation}\label{eq:SN-even-decomposition}
\mathcal S_{N,\mathrm e}^*
=
\mathcal M_N^*
+
\mathcal S_{N,\mathrm e}^*(V=\square)
+
\mathcal S_{N,\mathrm e}^*(V\ne\square),
\end{equation}
where $\mathcal M_N^*$ is the first sum in \eqref{eq:SN-even-poisson}, while
$\mathcal S_{N,\mathrm e}^*(V=\square)$ and
$\mathcal S_{N,\mathrm e}^*(V\ne\square)$ are the second sum with
$s_{\mathrm e}(h;C)$ replaced by $s_{\mathrm e}(h;C)_{\square}$ and by
$s_{\mathrm e}(h;C)_{\ne\square}$, respectively.

We next consider the odd-degree contribution. Applying
\eqref{eq:poisson_odd} to the two sums over $H$ in
$\mathcal S_{N,\mathrm o}^*$ in the same way, the extra factor $\sqrt q$
raising the prefactor to $q^{2g+5/2}$, we obtain
\begin{equation}\label{eq:SN-odd-poisson}
\mathcal S_{N,\mathrm o}^*
=
q^{2g+5/2}
\sum_{\substack{h,m\in\A^+\\ d(h)+d(m)\le N\\ d(h)\text{ odd}}}
\frac{1}{|h|\,|m|^{3/2}}
\sum_{\substack{C\in\A_{\le g}^+\\C\mid h^\infty}}
\frac{s_{\mathrm o}(h;C)}{|C|^2}.
\end{equation}
Here $d(h)$ is odd, so both degrees in \eqref{eq:s-odd} are even and square
dual variables occur in both sums. Splitting
$s_{\mathrm o}(h;C)=s_{\mathrm o}(h;C)_{\square}+s_{\mathrm o}(h;C)_{\ne\square}$
as before gives
\begin{equation}\label{eq:SN-odd-decomposition}
\mathcal S_{N,\mathrm o}^*
=
\mathcal S_{N,\mathrm o}^*(V=\square)
+
\mathcal S_{N,\mathrm o}^*(V\ne\square),
\end{equation}
the two terms being \eqref{eq:SN-odd-poisson} with $s_{\mathrm o}(h;C)$
replaced by $s_{\mathrm o}(h;C)_{\square}$ and by
$s_{\mathrm o}(h;C)_{\ne\square}$, respectively.

Finally, Lemma~\ref{lem:AFE} and \eqref{eq:SN_def} give
$M^*(g) = \sum_{N\in\{g,g-1\}}\mathcal S_N^*$.
Summing \eqref{eq:parity_split} over $N\in\{g,g-1\}$ and using
\eqref{eq:SN-even-decomposition} and
\eqref{eq:SN-odd-decomposition}, we obtain, after renaming
$\varepsilon$,
\begin{equation}\label{eq:Mstar-decomposition}
M^*(g)
=
\mathcal M^*
+
\mathcal S^*(V=\square)
+
\mathcal S^*(V\ne\square)
+
O_\varepsilon\!\left(q^{g(1+\varepsilon)/2}\right),
\end{equation}
where
\begin{align}
\mathcal M^*
&:=
\sum_{N\in\{g,g-1\}}\mathcal M_N^*,   \label{eq:Mstar-zero}  \\
\mathcal S^*(V=\square)
&:=
\sum_{N\in\{g,g-1\}}
\left(
\mathcal S_{N,\mathrm e}^*(V=\square)
+
\mathcal S_{N,\mathrm o}^*(V=\square)
\right), \label{eq:Mstar-square} \\
\mathcal S^*(V\ne\square)
&:=
\sum_{N\in\{g,g-1\}}
\left(
\mathcal S_{N,\mathrm e}^*(V\ne\square)
+
\mathcal S_{N,\mathrm o}^*(V\ne\square)
\right). \label{eq:Mstar-nonsquare}
\end{align}

\section{Contribution from the zero dual variable}\label{sec:main-term}

The purpose of this section is to evaluate the zero-dual contribution
$\mathcal M^*$ of \eqref{eq:Mstar-decomposition}, which produces the principal
term of Theorem~\ref{thm:main}.

Put
\begin{equation}\label{eq:C-def}
\mathcal C^*(v)
:=
\prod_{P\in\mathcal P}
\left(
1-\frac{v^{d(P)}}{|P|(|P|+1)}
\right).
\end{equation}
Since $q^{d(P)}=|P|$, this is the Euler product $\mathcal C$ of
\cite[(5.4)]{Fl17-1} and \cite[(4.2)]{AM-21} in the rescaled variable:
\begin{equation}\label{eq:C-vs-Florea}
\mathcal C^*(v)=\mathcal C(v/q),
\qquad\text{where}\quad
\mathcal C(u)=\prod_{P\in\mathcal P}\left(1-\frac{u^{d(P)}}{|P|+1}\right)
\end{equation}
(in \cite{Ju20} both $\mathcal C$ and $\mathscr P$ are denoted $C$, the two
being the same product in the variables $u$ and $s$). The rescaling
is chosen so that $\mathcal C^*$ is the Euler product $\mathscr P$ of
\S\ref{sec:intro} in the variable $v=q^{1-s}$:
\begin{equation}\label{eq:C-vs-P}
\mathscr P(s)=\mathcal C^*\bigl(q^{1-s}\bigr),
\qquad
\mathcal C^*(1)=\mathscr P(1).
\end{equation}
By \eqref{eq:C-vs-Florea} the properties of $\mathcal C$ recorded in
\cite[(5.4)--(5.5)]{Fl17-1} transfer at once: the product $\mathcal C^*$
converges absolutely for $|v|<q$, and there
\begin{equation}\label{eq:Cfact}
\mathcal C^*(v)=\left(1-\frac{v}{q}\right)\mathcal K(v),
\qquad\text{where}\quad
\mathcal K(v):=\prod_{P\in\mathcal P}
\left(1+\frac{(v/q)^{d(P)}}{(1+|P|)\bigl(|P|-(v/q)^{d(P)}\bigr)}\right).
\end{equation}
The product $\mathcal K$ converges absolutely for $|v|<q^2$, so
$(1-v/q)\mathcal K(v)$ continues $\mathcal C^*$ holomorphically to
$|v|<q^2$; from here on $\mathcal C^*$ denotes this continuation.
Equivalently,
\begin{equation}\label{eq:C-over-qv}
\frac{\mathcal C^*(v)}{q-v}=\frac{\mathcal K(v)}{q}
\qquad(|v|<q^2),
\end{equation}
the apparent singularity on the left at $v=q$ being removable.

\begin{proposition}\label{prop:principal}
For every $\varepsilon>0$ and every $1<\rho<q$, one has
\begin{equation}\label{eq:Mstar-zero-decomposition}
\mathcal M^*
=
\mathcal Q_g
+
\mathcal R_0(g;\rho)
+
O_\varepsilon\!\left(q^{\varepsilon g}\right),
\end{equation}
where
\begin{equation}\label{eq:P-final}
\mathcal Q_g
:=
\frac{q^{2g+2}}{\zA(2)}\,
\frac{\mathscr P(1)}{2(1-q^{-1/2})}
\left[
(2g+2)
+\frac{4}{\log q}\frac{\mathscr P'}{\mathscr P}(1)
+2\zA\!\left(\tfrac12 \right)
\right]
\end{equation}
and
\begin{equation}\label{eq:R0}
\mathcal R_0(g;\rho)
:=
\frac{q^{2g+2}}{\zA(2)}
\sum_{N\in\{g,g-1\}}
\frac{1}{2\pi i}
\oint_{|u|=\rho}
\frac{\mathcal C^*(u^2)\,du}
{(1-q^{-1/2}u)(1-u)^2(1+u)\,u^{N+1}}.
\end{equation}
\end{proposition}

\subsection{The principal generating function}

For $|u|<1$, define
\begin{equation}\label{eq:A-def}
\mathcal A(u)
:=
\sum_{l,m\in\A^+}
\frac{\phi(l^2)}{|l|^3|m|^{3/2}}
\prod_{P\mid l}
\left(1-\frac{1}{|P|^2}\right)^{-1}
u^{2d(l)+d(m)}.
\end{equation}
Since $2d(l)+d(m)=d(l^2m)$, grouping the terms of \eqref{eq:A-def}
according to $f=l^2m$ gives
$\mathcal A(u)=\sum_{f\in\A^+}a(f)u^{d(f)}$, where
$$
a(f):=
\sum_{\substack{l,m\in\A^+\\ l^2m=f}}
\frac{\phi(l^2)}{|l|^3|m|^{3/2}}
\prod_{P\mid l}\left(1-\frac{1}{|P|^2}\right)^{-1}.
$$
We now evaluate $\mathcal A(u)$ in terms of the Euler product $\mathcal C^*$.

\begin{lemma}
The series $\mathcal A(u)$ converges absolutely for $|u|<1$, and there
\begin{equation}\label{eq:A}
\mathcal A(u)
=
\frac{\mathcal C^*(u^2)}
{(1-q^{-1/2}u)(1-u^2)}.
\end{equation}
Moreover $\mathcal K$ is nonvanishing on $|v|<q^2$, so that the only zero of
$\mathcal C^*$ on $|v|<q^2$ is the simple zero at $v=q$. Consequently
\begin{equation}\label{eq:A-continuation}
\mathcal A(u)
=
\frac{(1+q^{-1/2}u)\mathcal K(u^2)}
{1-u^2},
\end{equation}
and $\mathcal A$ continues meromorphically to $|u|<q$, its only poles there
being simple poles at $u=\pm1$.
\end{lemma}

\begin{proof}
Since $\phi(l^2)\le|l|^2$ and
$\prod_{P\mid l}\bigl(1-|P|^{-2}\bigr)^{-1}\le\zA(2)$, the terms of
\eqref{eq:A-def} are dominated in modulus by
$\zA(2)|l|^{-1}|m|^{-3/2}|u|^{2d(l)+d(m)}$, whose sum over $l,m\in\A^+$ is
finite for $|u|<1$. Hence $\mathcal A(u)$ converges absolutely for $|u|<1$
and the double sum factors as
$$
\mathcal A(u)
=\Biggl(\sum_{l\in\A^+}\frac{\phi(l^2)}{|l|^3}
\prod_{P\mid l}\left(1-\frac{1}{|P|^2}\right)^{-1}u^{2d(l)}\Biggr)
\Biggl(\sum_{m\in\A^+}\frac{u^{d(m)}}{|m|^{3/2}}\Biggr).
$$
The $m$-factor is $(1-q^{-1/2}u)^{-1}$; for the $l$-factor, the identity
$$
\frac{\phi(l^2)}{|l|^3}\prod_{P\mid l}\left(1-\frac{1}{|P|^2}\right)^{-1}
=\frac{1}{|l|}\prod_{P\mid l}\frac{|P|}{|P|+1},
$$
in which both sides are multiplicative and agree at $l=P^j$, together with
$u^{2d(l)}/|l|=(u^2/q)^{d(l)}$, gives
\begin{align*}
\sum_{l\in\A^+}\frac{\phi(l^2)}{|l|^3}
\prod_{P\mid l}\left(1-\frac{1}{|P|^2}\right)^{-1}u^{2d(l)}
&=\sum_{l\in\A^+}\left(\frac{u^2}{q}\right)^{d(l)}
\prod_{P\mid l}\frac{|P|}{|P|+1}\\
&=\mathcal C\!\left(\frac{u^2}{q}\right)
 \mathcal Z_{\A}\!\left(\frac{u^2}{q}\right)
=\frac{\mathcal C^*(u^2)}{1-u^2}.
\end{align*}
The second equality is \cite[Section~5]{Fl17-1} (see also
\cite[(3.15)]{Ju20} and \cite[(4.10)]{AM-21}); the third holds by
\eqref{eq:C-vs-Florea} and $\mathcal Z_{\A}(u^2/q)=(1-u^2)^{-1}$. This
proves \eqref{eq:A}.

The nonvanishing of $\mathcal K$ on $|v|<q^2$ follows from
\eqref{eq:Cfact}: no local factor vanishes there, since
$\bigl|(v/q)^{d(P)}\bigr|<|P|$, and the product converges absolutely.
Substituting \eqref{eq:Cfact} into \eqref{eq:A} gives
$\mathcal A(u)=(1+q^{-1/2}u)\mathcal K(u^2)(1-u^2)^{-1}$ for $|u|<1$, and
the right-hand side is meromorphic on $|u|<q$ with only the simple poles
$u=\pm1$. This proves \eqref{eq:A-continuation}.
\end{proof}

\subsection{Proof of Proposition~\ref{prop:principal}}

By \eqref{eq:Mstar-zero}, $\mathcal M^*=\sum_{N\in\{g,g-1\}}\mathcal M_N^*$.
Note first that $\mathcal R_0(g;\rho)$ is well defined for every $1<\rho<q$:
by \eqref{eq:Cfact} each of the two integrands in \eqref{eq:R0}
satisfies
\begin{equation}\label{eq:R0-integrand}
\frac{\mathcal C^*(u^2)}{(1-q^{-1/2}u)(1-u)^2(1+u)u^{N+1}}
=\frac{(1+q^{-1/2}u)\mathcal K(u^2)}{(1-u)^2(1+u)u^{N+1}}
\qquad(N\in\{g,g-1\}),
\end{equation}
and the right-hand side is holomorphic on $0<|u|<q$ apart from the poles at
$u=\pm1$, the apparent singularity at $u=q^{1/2}$ being removable.

\begin{proof}
By Lemma~\ref{lem:Gc-values}(3) one has
$\Gc(0,\psi_h)=\phi(h)\mathbf 1_{h=\square}$; in particular the zero dual
variable occurs only for square $h$, so that $d(h)$ is even, $\nu(h)=0$, and
the phase normalization of \S\ref{sec:poisson} plays no role in this section.
Therefore
$$
\mathcal M_N^*
=\frac{q^{2g+2}}{\zA(2)}
\sum_{\substack{h,m\in\A^+\\ d(h)+d(m)\le N\\ h=\square}}
\frac{\phi(h)}{|h|^{3/2}|m|^{3/2}}
\sum_{\substack{C\in\A^+_{\le g}\\ C\mid h^\infty}}\frac{1}{|C|^2}.
$$
As in \cite[Section~5]{Fl17-1}, the inner $C$-sum may be extended to all
$C\mid h^\infty$: uniformly for $d(h)\le g$,
$$
\sum_{\substack{C\in\A^+_{\le g}\\ C\mid h^\infty}}\frac{1}{|C|^2}
=\prod_{P\mid h}\left(1-\frac{1}{|P|^2}\right)^{-1}
+O_\varepsilon\!\left(q^{-(2-\varepsilon)g}\right),
$$
and the resulting error contributes $O_\varepsilon(q^{\varepsilon g})$.
Writing $h=l^2$, this gives
\begin{equation}\label{eq:MN-star}
\mathcal M_N^*
=\frac{q^{2g+2}}{\zA(2)}\sum_{\substack{l,m\in\A^+\\2d(l)+d(m)\le N}}
\frac{\phi(l^2)}{|l|^3|m|^{3/2}}
\prod_{P\mid l}\left(1-\frac{1}{|P|^2}\right)^{-1}
+O_\varepsilon\!\left(q^{\varepsilon g}\right).
\end{equation}
Since $\mathcal A(u)$ converges absolutely for $|u|<1$, the cumulative
Perron formula \eqref{eq:perron_partial}, applied to the main sum in
\eqref{eq:MN-star} with $0<r<1$, gives
\begin{equation}\label{eq:MN-star-perron}
\mathcal M_N^*
=\frac{q^{2g+2}}{\zA(2)}\frac{1}{2\pi i}
\oint_{|u|=r}\frac{\mathcal A(u)}{(1-u)u^{N+1}}\,du
+O_\varepsilon\!\left(q^{\varepsilon g}\right).
\end{equation}
Fix $1<\rho<q$.  By \eqref{eq:A-continuation} the integrand of
\eqref{eq:MN-star-perron} is \eqref{eq:R0-integrand};
since $\mathcal K(u^2)$ is holomorphic and nonzero for $|u|<q$, it has a
double pole at $u=1$, a simple pole at $u=-1$ and no other pole in
$r<|u|<\rho$.  Shifting the contour from $|u|=r$ to $|u|=\rho$, we obtain
\begin{equation}\label{eq:MN-star-contour}
\begin{aligned}
\mathcal M_N^*
=
\frac{q^{2g+2}}{\zA(2)}
\frac{1}{2\pi i}
\oint_{|u|=\rho}
\frac{\mathcal A(u)}{(1-u)u^{N+1}}\,du
-
\frac{q^{2g+2}}{\zA(2)}
\left(
\operatorname{Res}_{u=1}
+
\operatorname{Res}_{u=-1}
\right)
\frac{\mathcal A(u)}{(1-u)u^{N+1}}
+
O_\varepsilon\!\left(q^{\varepsilon g}\right).
\end{aligned}
\end{equation}

For the residue calculations we use \eqref{eq:A}, which holds on
$0<|u|<q$ by meromorphic continuation:
$$
\frac{\mathcal A(u)}{(1-u)u^{N+1}}
=
\frac{\mathcal C^*(u^2)u^{-N-1}}
{(1-q^{-1/2}u)(1-u)^2(1+u)},
$$
with a double pole at $u=1$ and a simple pole at $u=-1$. Writing
$G(u):=\mathcal C^*(u^2)u^{-N-1}\bigl((1-q^{-1/2}u)(1+u)\bigr)^{-1}$, so
that $\mathcal A(u)\bigl((1-u)u^{N+1}\bigr)^{-1}=G(u)(1-u)^{-2}$, one has
$$
\operatorname{Res}_{u=1}\frac{\mathcal A(u)}{(1-u)u^{N+1}}=G'(1),
$$
and logarithmic differentiation of $G$ at $u=1$ gives
$$
-\operatorname{Res}_{u=1}
\frac{\mathcal A(u)}{(1-u)u^{N+1}}
=
\frac{\mathcal C^*(1)}{2(1-q^{-1/2})}
\left[
(N+1)+\frac12
-\frac{q^{-1/2}}{1-q^{-1/2}}
-2\frac{(\mathcal C^*)'(1)}{\mathcal C^*(1)}
\right],
$$
while at $u=-1$
$$
-\operatorname{Res}_{u=-1}
\frac{\mathcal A(u)}{(1-u)u^{N+1}}
=
\frac{(-1)^N\mathcal C^*(1)}
{4(1+q^{-1/2})}.
$$
Substituting these residue calculations into
\eqref{eq:MN-star-contour}, we obtain
\begin{align}\label{eq:MN-star-residues}
\mathcal M_N^*
&=
\frac{q^{2g+2}}{\zA(2)}
\Biggl\{
\frac{\mathcal C^*(1)}{2(1-q^{-1/2})}
\left[
(N+1)+\frac12
-\frac{q^{-1/2}}{1-q^{-1/2}}
-2\frac{(\mathcal C^*)'(1)}{\mathcal C^*(1)}
\right]
\nonumber\\
&\hspace{25mm}
+
\frac{(-1)^N\mathcal C^*(1)}
{4(1+q^{-1/2})}
+
\frac{1}{2\pi i}
\oint_{|u|=\rho}
\frac{\mathcal A(u)}
{(1-u)u^{N+1}}\,du
\Biggr\}
+
O_\varepsilon\!\left(q^{\varepsilon g}\right).
\end{align}

Summing \eqref{eq:MN-star-residues} over $N\in\{g,g-1\}$ and using
\eqref{eq:Mstar-zero}, the contributions from the residue at $u=-1$
cancel, since $(-1)^g+(-1)^{g-1}=0$. Moreover,
$\sum_{N\in\{g,g-1\}}(N+1)=2g+1$.
Thus the total contribution of the residues is
$$
\frac{\mathcal C^*(1)}{2(1-q^{-1/2})}
\left[
(2g+2)
-\frac{2q^{-1/2}}{1-q^{-1/2}}
-\frac{4(\mathcal C^*)'(1)}{\mathcal C^*(1)}
\right].
$$

By \eqref{eq:C-vs-P} one has $\mathcal C^*(1)=\mathscr P(1)$, and
logarithmic differentiation at $s=1$ gives
$\frac{\mathscr P'}{\mathscr P}(1)=-(\log q)\frac{(\mathcal C^*)'}{\mathcal C^*}(1)$.
Moreover
$-\frac{q^{-1/2}}{1-q^{-1/2}}=\frac{1}{1-q^{1/2}}=\zA\!\left(\tfrac12\right)$.
It follows from \eqref{eq:P-final} that the total residue contribution,
multiplied by $q^{2g+2}/\zA(2)$, equals $\mathcal Q_g$. Moreover, by
\eqref{eq:A} the contour terms in the sum of \eqref{eq:MN-star-residues} add
up to $\mathcal R_0(g;\rho)$. This proves
\eqref{eq:Mstar-zero-decomposition}.
\end{proof}

We retain the displaced contour $\mathcal R_0(g;\rho)$ explicitly, rather
than estimating it, because the principal square-dual contribution produces
$-\mathcal R_0(g;\rho)$ and the two cancel exactly in \S\ref{sec:proof}.

\section{Contribution from square dual variables}\label{sec:square-dual}

The purpose of this section is to evaluate the square-dual contribution
$\mathcal S^*(V=\square)$ of \eqref{eq:Mstar-decomposition}.
Recall from \eqref{eq:Mstar-square} that
\begin{equation}\label{eq:Sstar-square-split}
\mathcal{S}^*(V=\square)
=\sum_{N\in\{g,g-1\}}\left(\mathcal{S}_{N,\mathrm e}^{*}(V=\square)+\mathcal{S}_{N,\mathrm o}^{*}(V=\square)\right).
\end{equation}

\begin{proposition}\label{prop:square-dual-main}
For every $\varepsilon>0$ and every $1<\rho<q$, one has
\begin{align}\label{eq:square-dual-main}
\mathcal{S}^*(V=\square)
=-\mathcal R_0(g;\rho)+\mathcal S_0(g)
+O_\varepsilon\!\left(q^{g(1+\varepsilon)/2}\right),
\end{align}
where $\mathcal R_0(g;\rho)$ is as in \eqref{eq:R0} and $\mathcal S_0(g)$ as in
\eqref{eq:S0-secondary}, evaluated explicitly in
Corollary~\ref{cor:S0-periodic}.
\end{proposition}

\subsection{The generating series \texorpdfstring{$\mathcal B$}{B}}

We first record the Euler-product input that governs both parity sectors.
For $h\in\A^+$, define
\begin{equation}\label{eq:Ah_def}
A_h(z):=\sum_{Y\in\A^+}z^{d(Y)}\frac{\Gc(Y^2,\psi_h)}{\sqrt{|h|}}.
\end{equation}
This is the phase-normalized counterpart of the function $A_f(z)$ of
\cite[Section~3.2]{Ju20}.

\begin{lemma}
For every $h\in\A^+$, the series $A_h(z)$ converges absolutely
for $|z|<q^{-1}$, and
\begin{equation}\label{eq:Ah-uniform}
\bigl|A_h(z)\bigr|\le\frac{1}{1-q|z|}
\qquad(|z|<q^{-1}),
\end{equation}
uniformly in $h$.
\end{lemma}

\begin{proof}
For $h=\prod_{P\mid h}P^{i_P}\in\A^+$, put
$\operatorname{sq}(h):=\prod_{P\mid h}P^{\lfloor i_P/2\rfloor}$. We shall use
two auxiliary facts: $\Gc(Y^2,\psi_h)$ vanishes unless
$\operatorname{sq}(h)\mid Y$, and
$\bigl|\Gc(Y^2,\psi_h)\bigr|\le\sqrt{|h|}\,\bigl|\operatorname{sq}(h)\bigr|$.
Since $\Gc(Y^2,\psi_h)$ is multiplicative in $h$ by \eqref{eq:Gc-mult}, as
are $|h|$ and $\operatorname{sq}(h)$, it suffices to prove both facts for
$h=P^i$ with $i\ge1$.

So let $h=P^i$ and write $Y=P^vY_0$ with $v\ge0$ and $P\nmid Y_0$, so that
$Y^2=P^{2v}Y_0^2$. Applying \eqref{eq:Gc-values} with $k=2v$, the value
$\Gc(Y^2,\psi_{P^i})$ vanishes unless $i\le2v$ with $i$ even, or $i=2v+1$.
In both surviving cases $i\le2v+1$, that is, $v\ge\lfloor i/2\rfloor$, which
says exactly that $\operatorname{sq}(P^i)\mid Y$; this is the first fact for
$h=P^i$. In those same cases
$$
\frac{\bigl|\Gc(Y^2,\psi_{P^i})\bigr|}{|P|^{i/2}}
=
\begin{cases}
\phi(P^i)|P|^{-i/2}<|P|^{i/2}=|P|^{\lfloor i/2\rfloor},
& i\le2v,\ i\text{ even},\\[4pt]
|P|^{i-1/2}|P|^{-i/2}=|P|^{(i-1)/2}=|P|^{\lfloor i/2\rfloor},
& i=2v+1,
\end{cases}
$$
which is the second fact for $h=P^i$.

Now let $M\ge0$. The number of $Y\in\A_M^+$ divisible by
$\operatorname{sq}(h)$ is $q^{M-d(\operatorname{sq}(h))}$ if
$d(\operatorname{sq}(h))\le M$, and zero otherwise, so the above two
auxiliary facts give
$$
\sum_{Y\in\A_M^+}\frac{\bigl|\Gc(Y^2,\psi_h)\bigr|}{\sqrt{|h|}}
\le\bigl|\operatorname{sq}(h)\bigr|q^{M-d(\operatorname{sq}(h))}=q^M .
$$
Grouping \eqref{eq:Ah_def} according to $d(Y)$ therefore bounds
$|A_h(z)|$ by $\sum_{M\ge0}q^M|z|^M=(1-q|z|)^{-1}$, which is finite for
$|z|<q^{-1}$; this proves both assertions.
\end{proof}

We also introduce the two-variable generating series that encodes the
$h$-sum after the $C$-extension. We denote its second variable by $w'$
to distinguish it from the variable $w$ in which the Perron formula
\eqref{eq:perron_partial} is applied below; the two variables will later be
related by $w'=w/(qz^{1/2})$.
For $q^{-2}<|z|<q^{-1}$ and $|w'|$ sufficiently small, define
\begin{equation}\label{eq:B_def}
\mathcal B(z,w')
:=
\sum_{h\in\A^+}
(w')^{d(h)}A_h(z)
\prod_{P\mid h}
\left(
1-\frac{1}{|P|^2z^{d(P)}}
\right)^{-1}.
\end{equation}

The series $\mathcal B(z,w')$ is the
phase-normalized counterpart of the series of \cite[Lemma~6.2]{Fl17-1}, and
the following lemma is the corresponding form of
\cite[Lemmas~6.2--6.3]{Fl17-1}. Its proof follows those lemmas, and
indicates where the phase normalization enters.

\begin{lemma}[Euler-product factorization of $\mathcal B$]\label{lem:B}
For $q^{-2}<|z|<q^{-1}$ and $|w'|$ sufficiently small, the series
$\mathcal B(z,w')$ converges absolutely, and it admits the following
factorization.

\begin{enumerate}[label=\textup{(\arabic*)}]
\item
One has
\begin{equation}\label{eq:B-factorization}
\mathcal{B}(z,w')
=
\mathcal Z_{\A}(z)
\mathcal Z_{\A}(w')
\mathcal Z_{\A}(q(w')^2z)
\prod_{P\in\mathcal P}\mathcal{B}_P(z,w'),
\end{equation}
where, writing
$Z_P:=z^{d(P)}$ and $W_P:=(w')^{d(P)}$ for $P\in\mathcal P$,
the local factor is
\begin{equation}\label{eq:BP}
\mathcal B_P(z,w')
=
1+
\frac{
W_P\bigl(1-|P|^2Z_P^2\bigr)
-|P|(|P|-1)Z_PW_P^2
-|P|Z_P\bigl(1-|P|Z_P\bigr)W_P^3
}{
|P|^2Z_P-1
}.
\end{equation}
Moreover, the Euler product
$\prod_{P\in\mathcal P}\mathcal B_P(z,w')$
converges absolutely under the weaker restriction
\begin{equation}\label{eq:B-region}
|w'|<q|z|,
\qquad
|w'|<q^{-1/2},
\qquad
|w'z|<q^{-1}.
\end{equation}

\item
For each $P\in\mathcal P$, define
\begin{equation}\label{eq:DP}
\mathcal D_P(z,w')
:=
\mathcal B_P(z,w')
\left(
1-\frac{W_P}{|P|^2Z_P}
\right)
(1-W_P^2)^{-1}.
\end{equation}
Then one has
\begin{align}\label{eq:B-continuation}
\mathcal{B}(z,w')
&=
\mathcal Z_{\A}(z)
\mathcal Z_{\A}(w')
\mathcal Z_{\A}(q(w')^2z)
\mathcal Z_{\A}\!\left(\frac{w'}{q^2z}\right)
\mathcal Z_{\A}((w')^2)^{-1}
\prod_{P\in\mathcal P}\mathcal{D}_P(z,w').
\end{align}
The right-hand side furnishes a meromorphic continuation of $\mathcal{B}(z,w')$
to the region
\begin{equation}\label{eq:D-region}
|z|>q^{-2},
\qquad
|w'|^2<q|z|,
\qquad
|w'|<q^3|z|^2,
\qquad
|w'|<1,
\qquad
|w'z|<q^{-1},
\end{equation}
where the Euler product
$\prod_{P\in\mathcal P}\mathcal{D}_P(z,w')$
converges absolutely and uniformly on compact subsets.
\end{enumerate}
\end{lemma}

\begin{proof}
We first show that $\mathcal B(z,w')$ converges
absolutely for $q^{-2}<|z|<q^{-1}$ and $|w'|$ sufficiently small, which
justifies the rearrangement into an Euler product made below. By
\eqref{eq:Ah-uniform} one has $|A_h(z)|\le(1-q|z|)^{-1}$ for $|z|<q^{-1}$,
uniformly in $h$, while $|z|>q^{-2}$ gives
$\prod_{P\mid h}\bigl|1-|P|^{-2}z^{-d(P)}\bigr|^{-1}
\le\bigl(1-q^{-2}/|z|\bigr)^{-d(h)}$. Since $\#\A_n^+=q^n$ for every
$n\ge0$, it follows that
$$
\mathcal B(z,w')\ll_z\sum_{n\ge0}
\left[q|w'|\left(1-\frac{q^{-2}}{|z|}\right)^{-1}\right]^n,
$$
which converges for $|w'|$ sufficiently small.

Fix $P\in\mathcal P$ and write $Y=P^vY_0$ with $v\ge0$ and $P\nmid Y_0$, so
that $Y^2=P^{2v}Y_0^2$ and $\bigl(\frac{Y_0^2}{P}\bigr)=1$. By
\eqref{eq:Gc-values} with $k=2v$ and $V_1=Y_0^2$, the values
$\Gc(Y^2,\psi_{P^i})$ entering the $P$-Euler factor vanish for every
$i\ge1$ except
$$
\Gc(Y^2,\psi_{P^{2j}})=\phi(P^{2j})
\quad(1\le j\le v),
\qquad
\Gc(Y^2,\psi_{P^{2v+1}})=|P|^{2v}\sqrt{|P|}.
$$
As in \cite[Lemma~6.2]{Fl17-1}, feeding them into the Euler-product
computation --- the finite sum over $j$ summed first, then the geometric
series over $v\ge0$, and the result multiplied by the local factors
$(1-Z_P)$, $(1-W_P)$ and $(1-|P|W_P^2Z_P)$ of $\mathcal Z_{\A}(z)^{-1}$,
$\mathcal Z_{\A}(w')^{-1}$ and $\mathcal Z_{\A}(q(w')^2z)^{-1}$,
respectively --- gives \eqref{eq:B-factorization} with the same local factor
\eqref{eq:BP}, and $\prod_{P}\mathcal B_P(z,w')$ converges absolutely in
\eqref{eq:B-region}.

Part~(2) corresponds to \cite[Lemma~6.3]{Fl17-1}. The local
factors at $P$ of $\mathcal Z_{\A}(w'/(q^2z))$ and
$\mathcal Z_{\A}((w')^2)^{-1}$ are
$\bigl(1-W_P/(|P|^2Z_P)\bigr)^{-1}$ and $1-W_P^2$, respectively, so
\eqref{eq:DP} gives
$$
\prod_{P\in\mathcal P}\mathcal B_P(z,w')
=
\mathcal Z_{\A}\!\left(\frac{w'}{q^2z}\right)
\mathcal Z_{\A}((w')^2)^{-1}
\prod_{P\in\mathcal P}\mathcal D_P(z,w'),
$$
which turns \eqref{eq:B-factorization} into \eqref{eq:B-continuation} in the
initial domain, and $\prod_{P}\mathcal D_P(z,w')$ converges absolutely and
uniformly on compact subsets of \eqref{eq:D-region}. Since the remaining
factors on the right-hand side of \eqref{eq:B-continuation} are meromorphic there, that
right-hand side provides the asserted meromorphic continuation of
$\mathcal B(z,w')$.
\end{proof}

Finally we record the value of the Euler product
$\prod_{P}\mathcal B_P(z,w')$ at $w'=q^{-1}$, in the variable $u$ with
$z=u^{-2}$, where it is given in terms of the Euler product
$\mathcal C^*$ of \eqref{eq:C-def}.

Put $\tilde u:=u^2/q$, so that $u^{-2}=1/(q\tilde u)$. By \eqref{eq:BP} the
factors $\mathcal B_P$ are those of \cite[Lemma~6.2]{Fl17-1}, so the
evaluation of $\prod_{P}\mathcal B_P\bigl(1/(q\tilde u),q^{-1}\bigr)$ in
\cite[Section~6]{Fl17-1}, and the absolute convergence shown there for
$q^{-1}<|\tilde u|<q$, apply here; in the variable $u$ that range is
$1<|u|<q$, and the evaluation reads, with $\mathcal C$ as in
\eqref{eq:C-vs-Florea},
$$
\prod_{P\in\mathcal P}\mathcal B_P(u^{-2},q^{-1})
=
\frac{\mathcal C(\tilde u)}{\zA(2)}\,
\frac{1-(q\tilde u)^{-1}}{1-\tilde u}.
$$
The only change here is the rescaling \eqref{eq:C-vs-Florea}, which turns
$\mathcal C(\tilde u)=\mathcal C(u^2/q)$ into $\mathcal C^*(u^2)$; together
with $1-(q\tilde u)^{-1}=1-u^{-2}$ and $1-\tilde u=1-u^2/q$ it gives
\begin{equation}\label{eq:B-principal}
\prod_{P\in\mathcal P}\mathcal{B}_P(u^{-2},q^{-1})
=
\frac{\mathcal C^*(u^2)}{\zA(2)}
\frac{q(u^2-1)}{u^2(q-u^2)}
\qquad(1<|u|<q).
\end{equation}
Both sides of \eqref{eq:B-principal} are holomorphic on $1<|u|<q$: by
\eqref{eq:C-over-qv} with $v=u^2$ its right-hand side equals
$\zA(2)^{-1}\mathcal K(u^2)(u^2-1)u^{-2}$, and $\mathcal K(u^2)$ is
holomorphic for $|u|<q$.

\subsection{Proof of Proposition~\ref{prop:square-dual-main}}

Fix $N\in\{g,g-1\}$. Restricting the two cumulative sums of
\eqref{eq:s-even} to the nonzero squares $V=Y^2$ and using
$\zA(2)^{-1}=(q-1)/q$ gives
\begin{equation}\label{eq:SN-even-square}
s_{\mathrm e}(h;C)_{\square}
=\frac{1}{\zA(2)}
\Biggl(q\sum_{Y\in\A^+_{\le\frac{d(h)}{2}-g+d(C)-2}}\frac{\Gc(Y^2,\psi_h)}{\sqrt{|h|}}
-\sum_{Y\in\A^+_{\le\frac{d(h)}{2}-g+d(C)-1}}\frac{\Gc(Y^2,\psi_h)}{\sqrt{|h|}}\Biggr),
\end{equation}
while the same restriction in \eqref{eq:s-odd} gives
\begin{equation}\label{eq:SN-odd-square}
s_{\mathrm o}(h;C)_{\square}
=\sum_{Y\in\A^+_{\frac{d(h)-3}{2}-g+d(C)}}\frac{\Gc(Y^2,\psi_h)}{\sqrt{|h|}}
-\frac{1}{q}\sum_{Y\in\A^+_{\frac{d(h)-1}{2}-g+d(C)}}\frac{\Gc(Y^2,\psi_h)}{\sqrt{|h|}}.
\end{equation}

We convert the
$Y$-sums in \eqref{eq:SN-even-square} and \eqref{eq:SN-odd-square} into
contour integrals over $|z|=q^{-1-\kappa}$, where $\varepsilon>0$ is fixed
and $0<\kappa<\min\{1/2,\varepsilon\}$. The even sector uses the
cumulative form \eqref{eq:perron_partial} of the function-field Perron
formula and the odd sector the exact form \eqref{eq:perron_exact}, applied
to the series $A_h(z)$ of \eqref{eq:Ah_def}; in each sector the two shifted
$Y$-sums combine into the single factor $qz-1$, and we obtain
\begin{align}
s_{\mathrm e}(h;C)_{\square}
&=\frac{1}{\zA(2)}\,\frac{1}{2\pi i}
\oint_{|z|=q^{-1-\kappa}}
\frac{(qz-1)A_h(z)}{(1-z)\,z^{\frac{d(h)}{2}-g+d(C)}}\,dz,
\label{eq:se-square-integral}\\
s_{\mathrm o}(h;C)_{\square}
&=\frac{1}{q}\,\frac{1}{2\pi i}
\oint_{|z|=q^{-1-\kappa}}
\frac{(qz-1)A_h(z)}{z^{\frac{d(h)+1}{2}-g+d(C)}}\,dz.
\label{eq:so-square-integral}
\end{align}
From here on the two sectors run in parallel, and we carry them together.
For $\bullet\in\{\mathrm e,\mathrm o\}$, let $\nu_\bullet\in\{0,1\}$ denote
the value taken on the sector $\bullet$ by the parity $\nu(h)$ of
\S\ref{sec:poisson}, so that $\nu_{\mathrm e}=0$ and $\nu_{\mathrm o}=1$, and
put
$$
c_{\mathrm e}:=\frac{q^{2g+2}}{\zA(2)},
\qquad
c_{\mathrm o}:=q^{2g+3/2},
\qquad
\Theta_\bullet(z):=\frac{c_\bullet z^g(qz-1)}{(1-z)^{1-\nu_\bullet}}.
$$
Inserting \eqref{eq:se-square-integral} into \eqref{eq:SN-even-poisson} and
\eqref{eq:so-square-integral} into \eqref{eq:SN-odd-poisson}, whose
prefactors $q^{2g+2}$ and $q^{2g+5/2}$ are absorbed into $c_{\mathrm e}$ and
$c_{\mathrm o}$, gives, for $\bullet\in\{\mathrm e,\mathrm o\}$,
\begin{equation}\label{eq:SN-square-integral}
\mathcal{S}_{N,\bullet}^{*}(V=\square)
=
\frac{1}{2\pi i}
\oint_{|z|=q^{-1-\kappa}}
\Theta_\bullet(z)
\sum_{\substack{h,m\in\A^+\\
d(h)+d(m)\le N\\
\nu(h)=\nu_\bullet}}
\frac{A_h(z)}{|h|\,|m|^{3/2}z^{\frac{d(h)+\nu_\bullet}{2}}}
\sum_{\substack{C\in\A^+_{\le g}\\ C\mid h^\infty}}
\frac{1}{|C|^2z^{d(C)}}\,dz.
\end{equation}

We extend the $C$-sum in \eqref{eq:SN-square-integral} to its full Euler
product: for each
$h\in\A^+$,
\begin{equation}\label{eq:C-extension-square}
\sum_{\substack{C\in\A^+_{\le g}\\ C\mid h^\infty}}
\frac{1}{|C|^2z^{d(C)}}
=
\prod_{P\mid h}
\left(
1-\frac{1}{|P|^2z^{d(P)}}
\right)^{-1}
-
\sum_{\substack{C\mid h^\infty\\ d(C)\ge g+1}}
\frac{1}{|C|^2z^{d(C)}}.
\end{equation}
One input changes in bounding the tail: \eqref{eq:Ah-uniform} supplies the
pointwise bound $|A_h(z)|\le(1-q|z|)^{-1}=(1-q^{-\kappa})^{-1}\ll_\kappa1$,
uniform in $h$, in place of the average used in \cite[Section~6]{Fl17-1}.
With it, for sufficiently small $\delta>0$ the divisor bound for
polynomials supported on the prime divisors of $h$ gives, on
$|z|=q^{-1-\kappa}$,
$$
\sum_{\substack{C\mid h^\infty\\d(C)\ge g+1}}\frac{1}{|C|^2|z|^{d(C)}}
=\sum_{\substack{C\mid h^\infty\\d(C)\ge g+1}}q^{-(1-\kappa)d(C)}
\ll_{\kappa,\delta}q^{-(1-\kappa)g+\delta(g+d(h))},
$$
and, using $d(h)+d(m)\le N\le g$ and $\sum_{m\in\A^+}|m|^{-3/2}\ll_q1$, the
tail contributes $\ll_{\kappa,\delta}q^{((1+\kappa)/2+2\delta)g}$ in
either parity sector. Choosing $\delta>0$ with
$\kappa+4\delta<\varepsilon$ makes
this $O_\varepsilon\bigl(q^{g(1+\varepsilon)/2}\bigr)$.
After the extension \eqref{eq:C-extension-square} the integrands are
holomorphic in the annulus
$q^{-3/2}\le|z|\le q^{-1-\kappa}$, since $A_h(z)$ is holomorphic for
$|z|<q^{-1}$ while the possible poles of
$\prod_{P\mid h}(1-|P|^{-2}z^{-d(P)})^{-1}$ lie on $|z|=q^{-2}$. The
$z$-contour may therefore be moved to $|z|=q^{-3/2}$ without crossing a
pole, and we obtain
\begin{equation}\label{eq:SN-square-fullC}
\mathcal S_{N,\bullet}^{*}(V=\square)
=
\frac{1}{2\pi i}
\oint_{|z|=q^{-3/2}}
\Theta_\bullet(z)\widetilde H_N^{\bullet}(z)\,dz
+
O_\varepsilon\!\left(q^{g(1+\varepsilon)/2}\right),
\end{equation}
where
$$
\widetilde H_N^{\bullet}(z)
:=
\sum_{\substack{h,m\in\A^+\\
d(h)+d(m)\le N\\
\nu(h)=\nu_\bullet}}
\frac{A_h(z)}{|h|\,|m|^{3/2}z^{\frac{d(h)+\nu_\bullet}{2}}}
\prod_{P\mid h}
\left(
1-\frac{1}{|P|^2z^{d(P)}}
\right)^{-1}.
$$

In $\widetilde H_N^{\bullet}$ the constraint $d(h)+d(m)\le N$ is on the
total degree, so $w$ tracks $d(h)+d(m)$ and cannot
select the parity of $d(h)$ alone; that parity has to be extracted from
$\mathcal B$ instead. Accordingly, put
$$
\mathcal B^{\bullet}(z,w')
:=
\tfrac12\bigl(\mathcal B(z,w')+(-1)^{\nu_\bullet}\mathcal B(z,-w')\bigr),
$$
which selects the terms of \eqref{eq:B_def} with $\nu(h)=\nu_\bullet$.

For the rest of this section, fix either square root $z^{1/2}$ of $z$ and
put $z^{3/2}:=zz^{1/2}$; all half-integral powers of $z$ refer to this
choice.
Set $w'=w/(qz^{1/2})$. Then $|h|=q^{d(h)}$ and $\#\A_n^+=q^n$ give
$$
(w')^{d(h)}=\frac{w^{d(h)}}{|h|z^{\frac{d(h)}{2}}}, \qquad
\sum_{m\in\A^+}\frac{w^{d(m)}}{|m|^{3/2}}=\frac{1}{1-q^{-1/2}w},
$$
respectively, the second being the free divisor's contribution.

Define
$$
K_{N,\bullet}(z,w)
:=
\frac{
z^{-\nu_\bullet/2}\,
\mathcal B^{\bullet}\!\left(z,\dfrac{w}{qz^{1/2}}\right)
}{
(1-q^{-1/2}w)(1-w)w^{N+1}
}.
$$
The cumulative Perron formula \eqref{eq:perron_partial} then gives
\begin{equation}\label{eq:Htilde-perron}
\widetilde H_N^{\bullet}(z)
=
\frac{1}{2\pi i}
\oint_{|w|=r}
K_{N,\bullet}(z,w)\,dw,
\end{equation}
where $r>0$ is sufficiently small; in particular, we may assume that
$r<q^{-3/4}$, which under $w'=w/(qz^{1/2})$ and $|z|=q^{-3/2}$ reads
$|w'|<q^{-1}$.

\begin{lemma}[Poles of $K_{N,\bullet}$]
\label{lem:K-poles}
For $|z|=q^{-3/2}$, $K_{N,\bullet}(z,\cdot)$ is meromorphic on the disc
$|w|<1$, with no poles there other than the five points $w=0$,
$w=\pm z^{1/2}$ of modulus $q^{-3/4}$, and $w=\pm q^{2}z^{3/2}$ of modulus
$q^{-1/4}$.
\end{lemma}

\begin{proof}
On $|z|=q^{-3/2}$ the substitution $w'=w/(qz^{1/2})$ gives
$|w'|=q^{-1/4}|w|$, and for $|w|<1$ the pair $(z,w')$ satisfies the five
conditions of \eqref{eq:D-region}, the binding one being $|w'|^2<q|z|$,
which reads $|w|<1$; these conditions involve $w'$ only through $|w'|$, so
they hold for $-w'$ as well, hence for both terms of
$\mathcal B^{\bullet}(z,w')$. By Lemma~\ref{lem:B}(2) the right-hand side of
\eqref{eq:B-continuation} therefore represents $\mathcal B(z,\pm w')$ on
$|w|<1$ as a meromorphic function whose only poles are those of the explicit
zeta factors. As the remaining factors of $K_{N,\bullet}$ are rational in
$w$, it follows that $K_{N,\bullet}(z,\cdot)$ is meromorphic on $|w|<1$.

It remains to locate its poles. Inserting \eqref{eq:B-continuation} into
$\mathcal B^{\bullet}$ and undoing the substitution $w=qz^{1/2}w'$ displays
every $w$-dependent factor of $K_{N,\bullet}$ at once: with
$\varrho\in\{\pm1\}$ indexing the two terms of $\mathcal B^{\bullet}$,
$$
K_{N,\bullet}(z,w)
=
\frac{z^{-\nu_\bullet/2}\,\mathcal Z_{\A}(z)
\left(1-\frac{w^2}{qz}\right)}
{2\,(1-w^2)(1-w)(1-q^{-1/2}w)\,w^{N+1}}
\sum_{\varrho\in\{\pm1\}}
\frac{\varrho^{\nu_\bullet}
\prod_{P\in\mathcal P}\mathcal D_P(z,\varrho w')}
{\left(1-\frac{\varrho w}{z^{1/2}}\right)
\left(1-\frac{\varrho w}{q^2z^{3/2}}\right)}.
$$
Here the four zeta factors of \eqref{eq:B-continuation} that involve $w$
have been evaluated as follows, the last column recording the moduli on
$|z|=q^{-3/2}$:
\begin{center}
\renewcommand{\arraystretch}{1.4}
\begin{tabular}{c c c c}
factor of \eqref{eq:B-continuation} & in the variable $w$ & singular at &
$|w|$\\
\hline
$\mathcal Z_{\A}(\varrho w')$
& $\left(1-\dfrac{\varrho w}{z^{1/2}}\right)^{-1}$
& $w=\varrho z^{1/2}$
& $q^{-3/4}$\\

$\mathcal Z_{\A}\!\left(\dfrac{\varrho w'}{q^2z}\right)$
& $\left(1-\dfrac{\varrho w}{q^2z^{3/2}}\right)^{-1}$
& $w=\varrho q^2z^{3/2}$
& $q^{-1/4}$\\

$\mathcal Z_{\A}\bigl(q(w')^2z\bigr)$
& $\left(1-w^2\right)^{-1}$
& $w=\pm1$
& $1$\\

$\mathcal Z_{\A}\bigl((w')^2\bigr)^{-1}$
& $1-\dfrac{w^2}{qz}$
& nowhere
& \\
\end{tabular}
\end{center}
The Euler product $\prod_P\mathcal D_P(z,\varrho w')$ is holomorphic on
$|w|<1$ by Lemma~\ref{lem:B}(2), and the remaining factors $(1-w)^{-1}$,
$(1-q^{-1/2}w)^{-1}$ and $w^{-(N+1)}$ of $K_{N,\bullet}$ are singular only
at $w=1$, $w=q^{1/2}$ and $w=0$, respectively. Of all these singularities
those at $w=\pm1$ and $w=q^{1/2}$ have modulus $\ge1$, so the only poles in
$|w|<1$ are $w=0$, $w=\pm z^{1/2}$ and $w=\pm q^{2}z^{3/2}$.
\end{proof}

\subsubsection{Shifting the $w$-contour}

We now enlarge the circle $|w|=r$ in \eqref{eq:Htilde-perron} to
$|w|=q^{-\beta}$: fix $\varepsilon>0$ and
$0<\beta<\min\{1/4,\varepsilon/2\}$, so that
$r<q^{-3/4}<q^{-1/4}<q^{-\beta}<1$. By Lemma~\ref{lem:K-poles} the
enlargement therefore crosses exactly the four nonzero poles
$w=\pm z^{1/2}$ and $w=\pm q^{2}z^{3/2}$, while $w=0$ is enclosed by both
contours.

Nothing depends on the choice of $z^{1/2}$: replacing it by $-z^{1/2}$
replaces $w'$ by $-w'$, which leaves $\mathcal B^{\mathrm e}$ unchanged and
changes the sign of both $z^{-1/2}$ and $\mathcal B^{\mathrm o}$, so
$K_{N,\bullet}$ is unchanged; and the two crossed pairs
$\{\pm z^{1/2}\}$ and $\{\pm q^2z^{3/2}\}$ are unchanged as sets, so the
residue sums over them are unchanged as well.

By Lemma~\ref{lem:K-poles} and the residue theorem, for
$\bullet\in\{\mathrm e,\mathrm o\}$ and $|z|=q^{-3/2}$,
\begin{equation}\label{eq:w-residue-thm}
\begin{aligned}
\frac{1}{2\pi i}\oint_{|w|=r}K_{N,\bullet}(z,w)\,dw
&=
\frac{1}{2\pi i}\oint_{|w|=q^{-\beta}}K_{N,\bullet}(z,w)\,dw\\
&\quad-
\sum_{\omega\in\{\pm z^{1/2}\}}
\operatorname{Res}_{w=\omega}K_{N,\bullet}(z,w)
-
\sum_{\omega\in\{\pm q^2z^{3/2}\}}
\operatorname{Res}_{w=\omega}K_{N,\bullet}(z,w).
\end{aligned}
\end{equation}
Substituting \eqref{eq:w-residue-thm} into
\eqref{eq:Htilde-perron} and the result into \eqref{eq:SN-square-fullC},
we obtain
$$
\mathcal S_{N,\bullet}^{*}(V=\square)
=
A_{N,\bullet}^{*}
+
B_{N,\bullet}^{*}
+
C_{N,\bullet}^{*}
+
O_\varepsilon\!\left(q^{g(1+\varepsilon)/2}\right),
$$
where
\begin{align}
A_{N,\bullet}^{*}
&:=
-\frac{1}{2\pi i}
\oint_{|z|=q^{-3/2}}
\Theta_\bullet(z)
\sum_{\omega\in\{\pm z^{1/2}\}}
\operatorname{Res}_{w=\omega}
K_{N,\bullet}(z,w)\,dz,
\label{eq:A-N}\\
B_{N,\bullet}^{*}
&:=
-\frac{1}{2\pi i}
\oint_{|z|=q^{-3/2}}
\Theta_\bullet(z)
\sum_{\omega\in\{\pm q^2z^{3/2}\}}
\operatorname{Res}_{w=\omega}
K_{N,\bullet}(z,w)\,dz,
\label{eq:B-N}\\
C_{N,\bullet}^{*}
&:=
\frac{1}{(2\pi i)^2}
\oint_{|z|=q^{-3/2}}\oint_{|w|=q^{-\beta}}
\Theta_\bullet(z)K_{N,\bullet}(z,w)\,dw\,dz.
\label{eq:C-N}
\end{align}
Thus $A_{N,\bullet}^{*}$ and $B_{N,\bullet}^{*}$ collect the residues at
$w=\pm z^{1/2}$ and at $w=\pm q^2z^{3/2}$, respectively, while
$C_{N,\bullet}^{*}$ is the integral over the enlarged contour; the minus
signs in \eqref{eq:A-N} and \eqref{eq:B-N} come from those in
\eqref{eq:w-residue-thm}.

The residue terms $A_{N,\bullet}^{*}$ and $B_{N,\bullet}^{*}$ are evaluated
in \S\ref{sec:principal-square} and \S\ref{sec:secondary-poles},
respectively. The enlarged-contour terms are error terms
(cf. \cite[(3.22)]{Ju20} and \cite[(5.8)~and~(5.12)]{AM-21}); we record the
verification because $K_{N,\bullet}$ here carries the additional factors
$(1-q^{-1/2}w)^{-1}$ and $(1-w)^{-1}$. On $|z|=q^{-3/2}$ and
$|w|=q^{-\beta}$ the four zeta factors of \eqref{eq:B-continuation} that
occur with positive power have arguments of moduli
$$
|z|=q^{-3/2},
\qquad
|\pm w'|=q^{-1/4-\beta},
\qquad
\bigl|q(w')^2z\bigr|=q^{-1-2\beta},
\qquad
\left|\frac{\pm w'}{q^2z}\right|=q^{-3/4-\beta},
$$
none of which equals $q^{-1}$ when $0<\beta<1/4$; with the locally uniform
convergence of $\prod_P\mathcal D_P$ on \eqref{eq:D-region} this gives
$\mathcal B^{\mathrm e},\mathcal B^{\mathrm o}\ll_\beta1$ there. The
remaining factors of $K_{N,\bullet}$ are bounded as well, $|1-w|$ and
$|1-q^{-1/2}w|$ being bounded away from zero on $|w|=q^{-\beta}$ while
$|w^{-(N+1)}|=q^{\beta(N+1)}$ and $|z^{-1/2}|=q^{3/4}$, so that
$K_{N,\bullet}(z,w)\ll_\beta q^{\beta(N+1)}$. On $|z|=q^{-3/2}$ one has
$\Theta_\bullet(z)\ll_q q^{2g+2}|z|^g$, the factors $(qz-1)$ and
$(1-z)^{\nu_\bullet-1}$ being $\ll_q1$ there; estimating the
$z$-integral in \eqref{eq:C-N} trivially, with $|z^g|=q^{-3g/2}$ and contour
length $\ll q^{-3/2}$, we obtain
$$
C_{N,\bullet}^{*}
\ll_\beta
q^{2g+2}\cdot q^{-3g/2}\cdot q^{-3/2}\cdot q^{\beta(N+1)}
\ll_\beta
q^{g/2+\beta N}.
$$
Since $N\le g$ and $\beta<\varepsilon/2$, it follows that
$C_{N,\bullet}^{*}=O_\varepsilon\bigl(q^{g(1+\varepsilon)/2}\bigr)$.
Consequently,
\begin{equation}\label{eq:SN-square-sector-AB}
\mathcal S_{N,\bullet}^{*}(V=\square)
=
A_{N,\bullet}^{*}
+
B_{N,\bullet}^{*}
+
O_\varepsilon\!\left(q^{g(1+\varepsilon)/2}\right).
\end{equation}
Summing \eqref{eq:SN-square-sector-AB} over the two sectors, we obtain
\begin{equation}\label{eq:SN-square-AB}
\begin{aligned}
\mathcal S_{N,\mathrm e}^{*}(V=\square)
+
\mathcal S_{N,\mathrm o}^{*}(V=\square)
=
A_{N,\mathrm e}^{*}
+
A_{N,\mathrm o}^{*}
+
B_{N,\mathrm e}^{*}
+
B_{N,\mathrm o}^{*}
+
O_\varepsilon\!\left(q^{g(1+\varepsilon)/2}\right).
\end{aligned}
\end{equation}

\subsubsection{Principal square-dual residues and parity matching}
\label{sec:principal-square}

This subsection evaluates $A_{N,\bullet}^{*}$, the contribution of the pole
family $w=\pm z^{1/2}$. The outcome, Proposition~\ref{prop:A-total}, is that
its total over $N\in\{g,g-1\}$ and over the two parity sectors equals
$-\mathcal R_0(g;\rho)$, and so cancels the contour integral retained in
\S\ref{sec:main-term}.

\begin{lemma}[Residues at $w=\pm z^{1/2}$]\label{lem:residue-first}
For $|z|=q^{-3/2}$, one has
\begin{equation}\label{eq:first-residue}
\sum_{\omega\in\{\pm z^{1/2}\}}
\operatorname{Res}_{w=\omega}K_{N,\bullet}(z,w)
=
-\frac{\mathcal Z_{\A}(z)\mathcal Z_{\A}(z/q)}
{2z^{(N+\nu_\bullet)/2}}
R_N^{\bullet}(z)
\prod_{P\in\mathcal P}\mathcal B_P(z,q^{-1}),
\end{equation}
where
$$
R_N^{\bullet}(z)
:=
\frac{1}
{(1-q^{-1/2}z^{1/2})(1-z^{1/2})}
+
\frac{(-1)^{N+\nu_\bullet}}
{(1+q^{-1/2}z^{1/2})(1+z^{1/2})}.
$$
\end{lemma}

\begin{proof}
Write $w'=w/(qz^{1/2})$, so that $w=\pm z^{1/2}$ corresponds to
$w'=\pm q^{-1}$. For $|z|=q^{-3/2}$ and $|w'|=q^{-1}$, the three conditions
in \eqref{eq:B-region} hold, so \eqref{eq:B-factorization} is available near
$w=\pm z^{1/2}$ and $\prod_P\mathcal B_P(z,w')$ is holomorphic there: the
pole is carried by $\mathcal Z_{\A}(w')$ alone. At $w=z^{1/2}$ this factor
is $\mathcal Z_{\A}(w/(qz^{1/2}))=(1-w/z^{1/2})^{-1}$, a simple pole of
residue $-z^{1/2}$, while the other three factors of
\eqref{eq:B-factorization}, namely $\mathcal Z_{\A}(z)$,
$\mathcal Z_{\A}(q(w')^2z)$ and $\prod_P\mathcal B_P(z,w')$, are evaluated
at $w'=q^{-1}$, where $q(w')^2z=z/q$. Hence
$$
\operatorname{Res}_{w=z^{1/2}}\mathcal B\!\left(z,\frac{w}{qz^{1/2}}\right)
=-z^{1/2}\,\Pi(z),
\qquad\text{where}\quad
\Pi(z):=\mathcal Z_{\A}(z)\mathcal Z_{\A}(z/q)
\prod_{P\in\mathcal P}\mathcal B_P(z,q^{-1}).
$$
At $w=-z^{1/2}$ the same computation applies to
$\mathcal B(z,-w/(qz^{1/2}))$: the singular factor is
$\mathcal Z_{\A}(-w/(qz^{1/2}))=(1+w/z^{1/2})^{-1}$, of residue
$+z^{1/2}$, whence
$\operatorname{Res}_{w=-z^{1/2}}\mathcal B(z,-w/(qz^{1/2}))=+z^{1/2}\Pi(z)$.

At each of the two points $w=\pm z^{1/2}$ only one of the two terms of
$\mathcal B^{\bullet}$ is singular, so
$$
\operatorname{Res}_{w=z^{1/2}}
\mathcal B^{\bullet}\!\left(z,\frac{w}{qz^{1/2}}\right)
=-\tfrac12z^{1/2}\Pi(z),
\qquad
\operatorname{Res}_{w=-z^{1/2}}
\mathcal B^{\bullet}\!\left(z,\frac{w}{qz^{1/2}}\right)
=(-1)^{\nu_\bullet}\tfrac12z^{1/2}\Pi(z),
$$
the factor $\tfrac12$ and the sign $(-1)^{\nu_\bullet}$ coming from the
definition of $\mathcal B^{\bullet}$. Since $(1-q^{-1/2}w)(1-w)w^{N+1}$ is
holomorphic and nonzero at $w=\pm z^{1/2}$, the definition of
$K_{N,\bullet}$ gives
$$
\operatorname{Res}_{w=\omega}K_{N,\bullet}(z,w)
=
\frac{z^{-\nu_\bullet/2}}
{(1-q^{-1/2}\omega)(1-\omega)\,\omega^{N+1}}
\operatorname{Res}_{w=\omega}
\mathcal B^{\bullet}\!\left(z,\frac{w}{qz^{1/2}}\right)
\qquad(\omega\in\{\pm z^{1/2}\}).
$$
Summing the two, and using
$\bigl(-z^{1/2}\bigr)^{-(N+1)}=(-1)^{N+1}z^{-(N+1)/2}$, we obtain
$$
\sum_{\omega\in\{\pm z^{1/2}\}}\operatorname{Res}_{w=\omega}K_{N,\bullet}(z,w)
=-\frac{z^{-\nu_\bullet/2}\Pi(z)}{2z^{N/2}}
\left\{\frac{1}{(1-q^{-1/2}z^{1/2})(1-z^{1/2})}
+\frac{(-1)^{N+\nu_\bullet}}{(1+q^{-1/2}z^{1/2})(1+z^{1/2})}\right\}.
$$
Here the sign $(-1)^{N+\nu_\bullet}$ originates in $w^{-(N+1)}$ at the
negative pole. The sum in braces is $R_N^{\bullet}(z)$, and
$z^{\nu_\bullet/2}z^{N/2}=z^{(N+\nu_\bullet)/2}$; this proves
\eqref{eq:first-residue}.
\end{proof}

We shall pass from $z$ to the variable $u$ with $z=u^{-2}$, in which
$\mathcal R_0(g;\rho)$ is defined in \eqref{eq:R0}. If $R>0$ and $G$ is
holomorphic in a neighborhood of $|z|=R^2$, then substituting $z=s^2$, the
circle $|s|=R$ covering $|z|=R^2$ twice, and then $s=u^{-1}$, which reverses
the orientation and so cancels the sign in $ds=-u^{-2}\,du$, gives
\begin{equation}\label{eq:z-to-u}
\frac{1}{2\pi i}\oint_{|z|=R^2}G(z)\,dz
=
\frac{1}{2\pi i}\oint_{|s|=R}G(s^2)s\,ds
=
\frac{1}{2\pi i}\oint_{|u|=R^{-1}}u^{-3}G(u^{-2})\,du.
\end{equation}

Substituting the residue formula of Lemma~\ref{lem:residue-first} into
\eqref{eq:A-N}, so that the minus signs of \eqref{eq:A-N} and
\eqref{eq:first-residue} cancel, and using $(qz-1)\mathcal Z_{\A}(z)=-1$ and
$\mathcal Z_{\A}(z/q)=(1-z)^{-1}$, which together give
$\Theta_\bullet(z)\mathcal Z_{\A}(z)\mathcal Z_{\A}(z/q)
=-c_\bullet z^{g}(1-z)^{\nu_\bullet-2}$, we obtain
$$
A_{N,\bullet}^{*}
=
-\frac{c_\bullet}{2}
\frac{1}{2\pi i}
\oint_{|z|=q^{-3/2}}
\frac{z^{g-\frac{N+\nu_\bullet}{2}}}{(1-z)^{2-\nu_\bullet}}
R_N^{\bullet}(z)
\prod_{P\in\mathcal P}\mathcal B_P(z,q^{-1})\,dz.
$$
We now convert this into an integral in the variable $u$: we apply
\eqref{eq:z-to-u} with $R=q^{-3/4}$, taking $z=u^{-2}$ and
$z^{1/2}=u^{-1}$; the new contour is $|u|=q^{3/4}$, which lies in the
annulus $1<|u|<q$ where \eqref{eq:B-principal} is available. Under this
substitution
$$
R_N^{\bullet}(u^{-2})
=
u^{2}
\left(
\frac{1}{(u-q^{-1/2})(u-1)}
+
\frac{(-1)^{N+\nu_\bullet}}{(u+q^{-1/2})(u+1)}
\right).
$$
Using \eqref{eq:B-principal} we obtain
\begin{equation}\label{eq:A-N-u}
A_{N,\bullet}^{*}
=
-\frac{q\,c_\bullet}{2\zA(2)}
\frac{1}{2\pi i}
\oint_{|u|=q^{3/4}}
\frac{u^{N-2g+1-\nu_\bullet}\,\mathcal C^*(u^2)}
{(u^2-1)^{1-\nu_\bullet}(q-u^2)}
\left(
\frac{1}{(u-q^{-1/2})(u-1)}
+
\frac{(-1)^{N+\nu_\bullet}}
{(u+q^{-1/2})(u+1)}
\right)du.
\end{equation}

\begin{proposition}
\label{prop:A-total}
For every $1<\rho<q$, one has
\begin{equation}\label{eq:A-total}
\sum_{N\in\{g,g-1\}}
\left(
A_{N,\mathrm e}^{*}
+
A_{N,\mathrm o}^{*}
\right)
=
-\mathcal R_0(g;\rho),
\end{equation}
with $\mathcal R_0(g;\rho)$ as in \eqref{eq:R0}.
\end{proposition}

\begin{proof}
By \eqref{eq:C-over-qv} with $v=u^2$, the integrand in \eqref{eq:A-N-u} is
holomorphic in the annulus $1<|u|<q$; its contour may therefore be moved to
$|u|=\rho$ for any $1<\rho<q$. We do so for each of the four pairs
$(N,\bullet)$ with $N\in\{g,g-1\}$ and $\bullet\in\{\mathrm e,\mathrm o\}$,
and add the two integrals in the definition \eqref{eq:R0} of
$\mathcal R_0(g;\rho)$. Since $(-1)^{N+\nu_\bullet}$ equals $(-1)^g$ for
$(g,\mathrm e)$ and $(g-1,\mathrm o)$ and $-(-1)^g$ for the other two, and
using $\zA(2)^{-1}=(q-1)/q$, the resulting six integrands --- four from
\eqref{eq:A-N-u} and two from \eqref{eq:R0} --- may be placed over a common
denominator. Factoring out $q^{2g+2}\mathcal C^*(u^2)u^{-g-1}$ and
treating the two parities of $g$ separately, a direct computation gives
\begin{align}\label{eq:A-total-difference}
\sum_{N\in\{g,g-1\}}\left(A_{N,\mathrm e}^{*}+A_{N,\mathrm o}^{*}\right)
+\mathcal R_0(g;\rho)
=\frac{q^{2g+2}}{2\pi i}\oint_{|u|=\rho}\frac{\mathcal C^*(u^2)}{u^{g+1}}\mathcal E_g(u)\,du,
\end{align}
where
$$
\mathcal E_g(u)
:=
\begin{cases}
\displaystyle
\frac{(q-1)u(u^2+1+2\sqrt q)}
{\sqrt q\,(q-u^2)(u^2-1)^2},
& g\ {\rm even},\\[8pt]
\displaystyle
\frac{(q-1)(\sqrt q(1+u^2)+2u^2)}
{\sqrt q\,(q-u^2)(u^2-1)^2},
& g\ {\rm odd}.
\end{cases}
$$

The integrand on the right of \eqref{eq:A-total-difference} is holomorphic
on $1<|u|<q$, since $(u^2-1)^{-2}$ is singular only on $|u|=1$ and
$\mathcal C^*(u^2)/(q-u^2)$ is holomorphic there by \eqref{eq:C-over-qv};
and it is an even function of
$u$, because $\mathcal C^*(u^2)$ is even while $u^{-g-1}$ and
$\mathcal E_g(u)$ are both odd when $g$ is even and both even when $g$ is
odd. The Laurent coefficients of an even function on an annulus satisfy
$c_n=(-1)^nc_n$, so $c_{-1}=0$ and the integral vanishes; thus
\eqref{eq:A-total-difference} yields \eqref{eq:A-total}.
\end{proof}

\subsubsection{Secondary square-dual poles}\label{sec:secondary-poles}

This subsection evaluates $B_{N,\bullet}^{*}$, the contribution of the pole
family $w=\pm q^2z^{3/2}$. The outcome, Proposition~\ref{prop:secondary-contour},
is that its total over $N\in\{g,g-1\}$ and over the two parity sectors
equals, up to an admissible error, the sum $\mathcal S_0(g)$ of the residues
at the three cube roots of $q^{-4}$ crossed when the $z$-contour is pushed
out; all three are retained here, whereas the earlier treatments keep only
the real one (see Remark~\ref{rem:comparison}).

\begin{lemma}[Residues at $w=\pm q^2z^{3/2}$]
For $|z|=q^{-3/2}$, one has
\begin{equation}\label{eq:second-residue}
\sum_{\omega\in\{\pm q^2z^{3/2}\}}
\operatorname{Res}_{w=\omega}K_{N,\bullet}(z,w)
=
-\frac{
\mathcal Z_{\A}(z)
\mathcal Z_{\A}(qz)
\mathcal Z_{\A}(q^3z^3)}
{2q^{2N}z^{(3N+\nu_\bullet)/2}\mathcal Z_{\A}(q^2z^2)}
Q_N^{\bullet}(z)
\prod_{P\in\mathcal P}\mathcal D_P(z,qz),
\end{equation}
where
\begin{equation}\label{eq:QN-def}
Q_N^{\bullet}(z)
:=
\frac{1}
{(1-q^{3/2}z^{3/2})(1-q^2z^{3/2})}
+
\frac{(-1)^{N+\nu_\bullet}}
{(1+q^{3/2}z^{3/2})(1+q^2z^{3/2})}.
\end{equation}
\end{lemma}

\begin{proof}
Set $a:=q^2z^{3/2}$, the value of $w$ at which $w'=w/(qz^{1/2})$ equals
$qz$. In \eqref{eq:B-continuation} the factor $\mathcal Z_{\A}(w'/(q^2z))$
becomes $(1-w/a)^{-1}$, a simple pole at $w=a$ of residue $-a$. On
$|z|=q^{-3/2}$ the pair $(z,qz)$ satisfies the five conditions of
\eqref{eq:D-region}. At $w=a$, that is $w'=qz$, the three remaining zeta
factors $\mathcal Z_{\A}(z)$, $\mathcal Z_{\A}(w')$ and
$\mathcal Z_{\A}(q(w')^2z)$ of \eqref{eq:B-continuation} have the arguments
$z$, $qz$ and $q^3z^3$, none of which equals $q^{-1}$; they are therefore
holomorphic there, and their value, together with
$\mathcal Z_{\A}((w')^2)^{-1}=\mathcal Z_{\A}(q^2z^2)^{-1}$ and
$\prod_P\mathcal D_P(z,qz)$, is
$$
\mathscr F(z):=\frac{\mathcal Z_{\A}(z)\mathcal Z_{\A}(qz)
\mathcal Z_{\A}(q^3z^3)}{\mathcal Z_{\A}(q^2z^2)}
\prod_{P\in\mathcal P}\mathcal D_P(z,qz).
$$
Hence $\operatorname{Res}_{w=a}\mathcal B(z,w/(qz^{1/2}))=-a\mathscr F(z)$,
and replacing $w$ by $-w$, which reverses the sign of the residue, gives
$\operatorname{Res}_{w=-a}\mathcal B(z,-w/(qz^{1/2}))=a\mathscr F(z)$.

The rest is the computation in the proof of Lemma~\ref{lem:residue-first},
with $z^{1/2}$ replaced by $a$ and $\Pi(z)$ by $\mathscr F(z)$: at each of
$w=\pm a$ only one of the two terms of $\mathcal B^{\bullet}$ is singular,
so its residue there is one half of the corresponding value above, that at
$w=-a$ carrying the extra factor $(-1)^{\nu_\bullet}$ from the definition of
$\mathcal B^{\bullet}$; and since $(1-q^{-1/2}w)(1-w)w^{N+1}$ is holomorphic
and nonzero at $w=\pm a$,
$$
\operatorname{Res}_{w=\omega}K_{N,\bullet}(z,w)
=
\frac{z^{-\nu_\bullet/2}}
{(1-q^{-1/2}\omega)(1-\omega)\,\omega^{N+1}}
\operatorname{Res}_{w=\omega}
\mathcal B^{\bullet}\!\left(z,\frac{w}{qz^{1/2}}\right)
\qquad(\omega\in\{\pm a\}).
$$
Summing the two, and using $(-a)^{-(N+1)}=(-1)^{N+1}a^{-(N+1)}$, we obtain
$$
\sum_{\omega\in\{\pm a\}}\operatorname{Res}_{w=\omega}K_{N,\bullet}(z,w)
=-\frac{z^{-\nu_\bullet/2}\mathscr F(z)}{2a^N}
\left\{\frac{1}{(1-q^{-1/2}a)(1-a)}
+\frac{(-1)^{N+\nu_\bullet}}{(1+q^{-1/2}a)(1+a)}\right\}.
$$
Since $q^{-1/2}a=q^{3/2}z^{3/2}$ the sum in braces is $Q_N^{\bullet}(z)$, and
$z^{\nu_\bullet/2}a^N=q^{2N}z^{(3N+\nu_\bullet)/2}$; this proves
\eqref{eq:second-residue}.
\end{proof}

The factor $(qz-1)$ carried by $\Theta_\bullet$ in \eqref{eq:B-N} combines
with the four zeta factors of \eqref{eq:second-residue}:
\begin{equation}\label{eq:zeta-secondary}
(qz-1)\mathcal Z_{\A}(z)
\frac{\mathcal Z_{\A}(qz)\mathcal Z_{\A}(q^3z^3)}
{\mathcal Z_{\A}(q^2z^2)}
=
-\frac{1-q^3z^2}
{(1-q^2z)(1-q^4z^3)}.
\end{equation}
Substituting \eqref{eq:second-residue} into \eqref{eq:B-N}, so that the
minus signs of \eqref{eq:B-N} and \eqref{eq:second-residue} cancel, and
using \eqref{eq:zeta-secondary}, we obtain
\begin{equation}\label{eq:B-N-explicit}
B_{N,\bullet}^{*}
=
-\frac{c_\bullet}{2q^{2N}}
\frac{1}{2\pi i}
\oint_{|z|=q^{-3/2}}
J_{N,\bullet}(z)\,dz,
\end{equation}
where
\begin{equation}\label{eq:J-N-def}
J_{N,\bullet}(z)
:=
z^{g-\frac{3N+\nu_\bullet}{2}}
\frac{(1-q^3z^2)\prod_{P\in\mathcal P}\mathcal D_P(z,qz)}
{(1-z)^{1-\nu_\bullet}(1-q^2z)(1-q^4z^3)}
Q_N^{\bullet}(z).
\end{equation}

Write $A:=q^{3/2}z^{3/2}$ and $B:=q^2z^{3/2}$, so that $A^2=q^3z^3$,
$B^2=q^4z^3$ and $AB=q^{7/2}z^3$. Putting the two fractions of
\eqref{eq:QN-def} over the common denominator $(1-A^2)(1-B^2)$ gives
\begin{align}
\frac{1}{(1-A)(1-B)}+\frac{1}{(1+A)(1+B)}
&=\frac{2(1+q^{7/2}z^3)}{(1-q^3z^3)(1-q^4z^3)},
\label{eq:secondary-plus}\\
\frac{1}{(1-A)(1-B)}-\frac{1}{(1+A)(1+B)}
&=\frac{2(q^{3/2}+q^2)z^{3/2}}{(1-q^3z^3)(1-q^4z^3)}.
\label{eq:secondary-minus}
\end{align}
Thus $Q_N^{\bullet}(z)$ is given by \eqref{eq:secondary-plus} when
$N+\nu_\bullet$ is even and by \eqref{eq:secondary-minus} when it is odd.
In either case the half-integral powers cancel against
$z^{g-(3N+\nu_\bullet)/2}$ in \eqref{eq:J-N-def}, because
$3N+\nu_\bullet\equiv N+\nu_\bullet\bmod 2$; thus $J_{N,\bullet}(z)$ is a
single-valued function of $z$.

For $q^{-2}<|z|<q^{-1}$, the substitution $w'=qz$ satisfies all the
conditions in \eqref{eq:D-region}, so
$\prod_{P\in\mathcal P}\mathcal D_P(z,qz)$ is holomorphic throughout this
annulus. After substituting \eqref{eq:secondary-plus} and
\eqref{eq:secondary-minus} into \eqref{eq:J-N-def}, the denominator of
$J_{N,\bullet}$ is a product of factors among $1-z$, $1-q^2z$, $1-q^3z^3$
and $(1-q^4z^3)^{2}$, the last squared because $(1-q^4z^3)^{-1}$ occurs both
in \eqref{eq:J-N-def} and in
\eqref{eq:secondary-plus}--\eqref{eq:secondary-minus}. The first three
vanish only at $z=1$, at $z=q^{-2}$ and on $|z|=q^{-1}$, none of which lies
in the open annulus, whereas the zeros of $1-q^4z^3$ are
$z_k:=q^{-4/3}e^{2\pi ik/3}$ ($k\in\{0,1,2\}$) with
$q^{-2}<|z_k|=q^{-4/3}<q^{-1}$. Thus $z_0$, $z_1$, $z_2$ are the only poles
of $J_{N,\bullet}$ in $q^{-2}<|z|<q^{-1}$, and each has order at most two,
the zeros of $1-q^4z^3$ being simple.

Define the total residue contribution at $z_0$, $z_1$, $z_2$ by
\begin{equation}\label{eq:S0-secondary}
\mathcal S_0(g)
:=
\sum_{N\in\{g,g-1\}}
\ \sum_{\bullet\in\{\mathrm e,\mathrm o\}}
\ \sum_{k=0}^2
\frac{c_\bullet}{2q^{2N}}
\operatorname{Res}_{z=z_k}J_{N,\bullet}(z).
\end{equation}

\begin{proposition}
\label{prop:secondary-contour}
For every $\varepsilon>0$, one has
\begin{equation}\label{eq:B-secondary-total}
\sum_{N\in\{g,g-1\}}
\left(
B_{N,\mathrm e}^{*}
+
B_{N,\mathrm o}^{*}
\right)
=
\mathcal S_0(g)
+
O_\varepsilon\!\left(q^{g(1+\varepsilon)/2}\right).
\end{equation}
\end{proposition}

\begin{proof}
Fix $\varepsilon>0$ and $0<\delta<\min\{1/3,\varepsilon\}$, and move the
contours in \eqref{eq:B-N-explicit} from $|z|=q^{-3/2}$ to
$|z|=q^{-1-\delta}$. Since
$q^{-3/2}<|z_k|=q^{-4/3}<q^{-1-\delta}<q^{-1}$, the only singularities of
$J_{N,\bullet}$ that this shift can cross are $z=z_k$ ($k\in\{0,1,2\}$); so by the
residue theorem and \eqref{eq:S0-secondary},
$$
\sum_{N\in\{g,g-1\}}\bigl(B_{N,\mathrm e}^{*}+B_{N,\mathrm o}^{*}\bigr)
=\mathcal S_0(g)
-\sum_{N\in\{g,g-1\}}\ \sum_{\bullet\in\{\mathrm e,\mathrm o\}}
\frac{c_\bullet}{2q^{2N}}\frac{1}{2\pi i}
\oint_{|z|=q^{-1-\delta}}J_{N,\bullet}(z)\,dz.
$$
Write $N=g-j$ with $j\in\{0,1\}$. On $|z|=q^{-1-\delta}$ the prefactor of
\eqref{eq:B-N-explicit}, combined with the explicit power of $z$ in
\eqref{eq:J-N-def}, is
$$
c_\bullet q^{-2N}|z|^{g-\frac{3N+\nu_\bullet}{2}}
=q^{(1+\delta)g/2+2+\delta\nu_\bullet/2+(1-3\delta)j/2},
$$
which for $0<\delta<1/3$, $j\in\{0,1\}$ and $\nu_\bullet\in\{0,1\}$ is
$\ll_\delta q^{(1+\delta)g/2}$. The circle $|z|=q^{-1-\delta}$ is
compactly contained in $q^{-2}<|z|<q^{-1}$, so
$\prod_{P}\mathcal D_P(z,qz)\ll_\delta1$ there by Lemma~\ref{lem:B}, the
remaining rational factors are holomorphic in a neighbourhood of the circle
by \eqref{eq:secondary-plus} and \eqref{eq:secondary-minus}, and the length
of the circle does not depend on $g$. The outer integrals therefore
contribute $O_\delta(q^{(1+\delta)g/2})$, and $\delta<\varepsilon$
gives \eqref{eq:B-secondary-total}.
\end{proof}

\subsubsection{Explicit evaluation of the secondary term}

This subsection evaluates the residue sum $\mathcal S_0(g)$ defined in
\eqref{eq:S0-secondary} and proves the explicit periodic formula for it
promised in Proposition~\ref{prop:square-dual-main}. Throughout we work in
the square-root coordinate
$z=s^2$,
which replaces the half-integral powers $z^{3/2}$ in $Q_N^{\bullet}(z)$ by
the ordinary powers $s^3$.

For each $k\in\{0,1,2\}$, choose the canonical lift
$s_k:=q^{-2/3}e^{-2\pi i k/3}$ of $z_k$, so that $s_k^2=z_k$ and
$s_k^3=q^{-2}$. The two lifts of $z_k$ are then $\varrho s_k$ with
$\varrho\in\{\pm1\}$, and $(\varrho s_k)^2=z_k$,
$(\varrho s_k)^3=\varrho q^{-2}$. This choice is used throughout the
subsection.

In the coordinate $s$ the two terms of \eqref{eq:QN-def} are indexed
uniformly by a sign $\eta\in\{\pm1\}$:
\begin{equation}\label{eq:QN-s}
Q_N^{\bullet}(s^2)=\sum_{\eta\in\{\pm1\}}
\frac{\eta^{N+\nu_\bullet}}{(1-\eta q^{3/2}s^3)(1-\eta q^2s^3)}.
\end{equation}

Since $z_k\ne0$, the map $s\mapsto s^2$ is locally biholomorphic at each of
the two lifts $\pm s_k$ of $z_k$, so for $F$ meromorphic near $z_k$ the change-of-variables formula for
residues gives
$\operatorname{Res}_{s=\varrho s_k}\bigl(2sF(s^2)\bigr)=\operatorname{Res}_{z=z_k}F(z)$
for either $\varrho$. Averaging over both lifts,
\begin{equation}\label{eq:residue-lift}
\operatorname{Res}_{z=z_k}F(z)
=\frac12\sum_{\varrho\in\{\pm1\}}
\operatorname{Res}_{s=\varrho s_k}\bigl(2sF(s^2)\bigr),
\end{equation}
a symmetric form that puts $\pm s_0$, $\pm s_1$, $\pm s_2$ on an equal
footing; it is what makes the residues at $s_k$ and at $-s_k$ agree in
Lemma~\ref{lem:mirror-cancellation-even}.

We now express $J_{N,\bullet}$, which appears in \eqref{eq:S0-secondary}, in
the square-root coordinate. From \eqref{eq:J-N-def} and \eqref{eq:QN-s},
\begin{equation}\label{eq:J-lifted}
2sJ_{N,\bullet}(s^2)
=
2s^{2g-3N+1-\nu_\bullet}
\frac{(1-q^3s^4)\prod_{P\in\mathcal P}\mathcal D_P(s^2,qs^2)}
{(1-s^2)^{1-\nu_\bullet}(1-q^2s^2)(1-q^4s^6)}
\sum_{\eta\in\{\pm1\}}
\frac{\eta^{N+\nu_\bullet}}
{(1-\eta q^{3/2}s^3)(1-\eta q^2s^3)}.
\end{equation}

We next determine which $\eta$-summand of \eqref{eq:J-lifted} can be
singular at $s=\varrho s_k$. Fix $k\in\{0,1,2\}$ and $\varrho\in\{\pm1\}$;
the sign $\eta$ continues to index the two summands in
\eqref{eq:J-lifted}, whereas $\varrho$ specifies one of the two lifts of
$z_k$.

The term of \eqref{eq:J-lifted} indexed by $\eta$ has, in the even sector,
denominator
\begin{equation}\label{eq:lifted-denominator}
(1-s^2)(1-q^2s^2)(1-q^4s^6)(1-\eta q^{3/2}s^3)(1-\eta q^2s^3),
\end{equation}
and in the odd sector the same denominator without the factor $1-s^2$.
Since $(\varrho s_k)^2=z_k$ and $(\varrho s_k)^3=\varrho q^{-2}$, among the
four factors of \eqref{eq:lifted-denominator} other than $1-q^4s^6$ only
$1-\eta q^2s^3$ can vanish at $s=\varrho s_k$, and it does so exactly when
$\eta=\varrho$. The remaining
factor $1-q^4s^6$ is common to the two $\eta$-terms and has a simple zero at
$\varrho s_k$. Thus for $\eta=\varrho$ two of the five factors in
\eqref{eq:lifted-denominator} vanish and the term has a pole of order at most
two, whereas for $\eta=-\varrho$ only $1-q^4s^6$ vanishes. That remaining
simple pole is spurious: adding the contributions of $N=g$ and $N=g-1$
produces a factor $1+\eta q^2s^3$, and
$(1+\eta q^2s^3)/(1-q^4s^6)=(1-\eta q^2s^3)^{-1}$, which equals $\tfrac12$ at
$s=\varrho s_k$ when $\eta=-\varrho$. This is carried out in
Lemma~\ref{lem:mirror-cancellation-even}; accordingly, for
$\eta\in\{\pm1\}$ put
\begin{equation}\label{eq:T-secondary}
T_\eta(s)
:=
\eta^g s^{-g}
\frac{q(1-q^3s^4)\prod_{P\in\mathcal P}\mathcal D_P(s^2,qs^2)}
{2(1-q^2s^2)(1-\eta q^{3/2}s^3)
(1-\eta q^2s^3)^2}
\left(
\frac{(q-1)s}{1-s^2}+\sqrt q\,\eta
\right).
\end{equation}
Since $\prod_{P}\mathcal D_P(s^2,qs^2)$ depends only on $s^2$, replacing $s$
by $-s$ and $\eta$ by $-\eta$ in \eqref{eq:T-secondary} gives
\begin{equation}\label{eq:T-mirror}
T_{-\eta}(-s)=-T_\eta(s).
\end{equation}

\begin{lemma}
\label{lem:mirror-cancellation-even}
One has
\begin{equation}\label{eq:S0-three-residues}
\mathcal S_0(g)
=
2\sum_{k=0}^{2}
\operatorname{Res}_{s=s_k}T_{+1}(s).
\end{equation}
\end{lemma}

\begin{proof}
Applying \eqref{eq:residue-lift} to the residues in \eqref{eq:S0-secondary}
and substituting \eqref{eq:J-lifted}, the
factor $\tfrac12$ of \eqref{eq:residue-lift} cancelling the leading $2$
there, we obtain by linearity of the residue
\begin{align}
\mathcal S_0(g)
={}&\sum_{k=0}^2\sum_{\varrho\in\{\pm1\}}\sum_{\eta\in\{\pm1\}}
\operatorname{Res}_{s=\varrho s_k}
\Biggl[\frac{(1-q^3s^4)\prod_{P\in\mathcal P}\mathcal D_P(s^2,qs^2)}
{(1-q^2s^2)(1-q^4s^6)(1-\eta q^{3/2}s^3)(1-\eta q^2s^3)}
\nonumber\\
&\hspace{8em}\times
\sum_{N\in\{g,g-1\}}\ \sum_{\bullet\in\{\mathrm e,\mathrm o\}}
\frac{c_\bullet}{2q^{2N}}
\frac{\eta^{N+\nu_\bullet}s^{2g-3N+1-\nu_\bullet}}
{(1-s^2)^{1-\nu_\bullet}}\Biggr].
\label{eq:S0-lifted-branches}
\end{align}
Fix $\eta$ and evaluate the double sum over $N\in\{g,g-1\}$ and
$\bullet\in\{\mathrm e,\mathrm o\}$ in \eqref{eq:S0-lifted-branches} term by
term; using $c_{\mathrm e}=q^{2g+2}/\zA(2)=q^{2g+1}(q-1)$,
$c_{\mathrm o}=q^{2g+3/2}$ and $\eta^2=1$, it equals
\begin{equation}\label{eq:two-truncations-combined}
\eta^gs^{-g}\frac{q}{2}\bigl(1+\eta q^2s^3\bigr)
\left(\frac{(q-1)s}{1-s^2}+\sqrt q\,\eta\right).
\end{equation}
Substituting \eqref{eq:two-truncations-combined} into
\eqref{eq:S0-lifted-branches} and cancelling the factor $1+\eta q^2s^3$
against $1-q^4s^6=(1-\eta q^2s^3)(1+\eta q^2s^3)$, we obtain
\begin{equation}\label{eq:S0-all-T-branches}
\mathcal S_0(g)=\sum_{k=0}^2\sum_{\varrho\in\{\pm1\}}\sum_{\eta\in\{\pm1\}}
\operatorname{Res}_{s=\varrho s_k}T_\eta(s),
\end{equation}
with $T_\eta$ as in \eqref{eq:T-secondary}.  Fix $k$ and $\varrho$.  Among
the denominator factors of $T_\eta$ only $1-\eta q^2s^3$ can vanish at
$s=\varrho s_k$, and it does so exactly when $\eta=\varrho$; hence no
denominator factor of $T_{-\varrho}$ is zero there, and the Euler product is
holomorphic there by Lemma~\ref{lem:B}, so $T_{-\varrho}$ contributes no
residue, whereas $T_\varrho$ carries the double-pole factor
$(1-\varrho q^2s^3)^{-2}$.  Thus
\eqref{eq:S0-all-T-branches} reduces to
$$
\mathcal S_0(g)
=
\sum_{k=0}^{2}
\left\{
\operatorname{Res}_{s=s_k}T_{+1}(s)
+
\operatorname{Res}_{s=-s_k}T_{-1}(s)
\right\}.
$$
Finally, the substitution $s\mapsto-s$ reverses the sign of a residue, and
$T_{-1}(-s)=-T_{+1}(s)$ by \eqref{eq:T-mirror}, so
$\operatorname{Res}_{s=-s_k}T_{-1}=\operatorname{Res}_{s=s_k}T_{+1}$;
substituting this into the display above gives
\eqref{eq:S0-three-residues}.
\end{proof}

We now evaluate the representative residues in
\eqref{eq:S0-three-residues}. Fix $k\in\{0,1,2\}$. Since $s_k^3=q^{-2}$,
one has $q^2z_k=q^2s_k^2=s_k^{-1}$.
Consequently,
\begin{equation}\label{eq:secondary-basic-identities}
s_k^{-g}=(q^2z_k)^g,
\qquad
1-q^{3/2}s_k^3=1-q^{-1/2}.
\end{equation}

Write $\mathscr D(z):=\prod_{P\in\mathcal P}\mathcal D_P(z,qz)$. The
residue calculation requires $\mathscr D'$ at $z_0$, $z_1$ and $z_2$, so we
record that $\mathscr D$ is holomorphic there. Each $z_k$ has
$|z_k|=q^{-4/3}$ and so lies in the annulus $q^{-2}<|z|<q^{-1}$, on which
$w'=qz$ was seen in \S\ref{sec:secondary-poles} to satisfy all five
conditions of \eqref{eq:D-region}; by Lemma~\ref{lem:B} the product
converges absolutely and uniformly on compact subsets of that annulus, and
$\mathscr D$ is therefore holomorphic near each of $z_0$, $z_1$ and $z_2$.

For $z^3=q^{-4}$, that is, for $z\in\{z_0,z_1,z_2\}$, define
\begin{equation}\label{eq:a-secondary}
\mathfrak a(z)
:=
-\frac{
(1-q^3z^2)
\bigl[q(q-1)z^2+q^{-1/2}(1-z)\bigr]}
{18z(1-q^2z)(1-z)}\,\mathscr D(z)
\end{equation}
and
\begin{align}
\mathfrak b(z)
:={}&
\left[
2+\frac{4q^3z^2}{1-q^3z^2}
-\frac{2q^2z}{1-q^2z}
-\frac{3}{\sqrt q-1}
\right]\mathfrak a(z)
+
\frac{
q(q-1)z(1+z)(1-q^3z^2)}
{18(1-q^2z)(1-z)^2}\,\mathscr D(z)
\nonumber\\
&+
\frac{
(1-q^3z^2)
\bigl[q(q-1)z^2+q^{-1/2}(1-z)\bigr]}
{9(1-q^2z)(1-z)}\,\mathscr D'(z).
\label{eq:b-secondary}
\end{align}
None of the denominator factors in \eqref{eq:a-secondary} and
\eqref{eq:b-secondary} vanishes at $z_0$, $z_1$ or $z_2$, so both
expressions are well defined there.

By \eqref{eq:secondary-basic-identities} the factor $1-q^{3/2}s^3$ of
\eqref{eq:T-secondary} takes the same nonzero value $1-q^{-1/2}$ at $s_0$,
$s_1$, $s_2$; the next lemma clears it by evaluating
$(1-q^{-1/2})\operatorname{Res}_{s=s_k}T_{+1}(s)$.

\begin{lemma}[Residue at a lifted secondary point]
\label{lem:secondary-residue}
For every $k\in\{0,1,2\}$, one has
\begin{equation}\label{eq:secondary-residue-evaluation}
(1-q^{-1/2})
\operatorname{Res}_{s=s_k}T_{+1}(s)
=
(q^2z_k)^g
\left(
\mathfrak a(z_k)g+\mathfrak b(z_k)
\right).
\end{equation}
\end{lemma}

\begin{proof}
Put
$$
R(s):=\frac{q(1-q^3s^4)}{2(1-q^2s^2)(1-q^{3/2}s^3)}
\left(\frac{(q-1)s}{1-s^2}+\sqrt q\right),
\qquad
\Delta(s):=1-q^2s^3,
$$
so that $T_{+1}(s)=G(s)\Delta(s)^{-2}$ with
$G(s):=s^{-g}R(s)\mathscr D(s^2)$.  Here $G$ is holomorphic near $s_k$,
since $\mathscr D$ is holomorphic near $z_k$ as shown above and no
denominator factor of $R$ vanishes at $s_k$.  Since $q^2s_k^3=1$ we have $\Delta'(s_k)=-3/s_k$ and
$\Delta''(s_k)=-6/s_k^2$; as $\Delta$ has a simple zero at $s_k$, expanding
$G$ and $\Delta$ in powers of $s-s_k$ gives
$$
\operatorname{Res}_{s=s_k}T_{+1}(s)
=\frac{G'(s_k)}{\Delta'(s_k)^2}
-\frac{G(s_k)\Delta''(s_k)}{\Delta'(s_k)^3}
=\frac{s_k^2}{9}G'(s_k)-\frac{2s_k}{9}G(s_k).
$$
By the product rule,
$sG'(s)=s^{-g}\bigl[(-gR(s)+sR'(s))\mathscr D(s^2)
+2s^2R(s)\mathscr D'(s^2)\bigr]$, so that, using
$s_k^{-g}=(q^2z_k)^g$ and $s_k^2=z_k$,
\begin{equation}\label{eq:secondary-residue-product-rule}
(1-q^{-1/2})\operatorname{Res}_{s=s_k}T_{+1}(s)
=(q^2z_k)^g\frac{(1-q^{-1/2})s_k}{9}
\Bigl[\bigl(-(g+2)R(s_k)+s_kR'(s_k)\bigr)\mathscr D(z_k)
+2z_kR(s_k)\mathscr D'(z_k)\Bigr].
\end{equation}
Using $s_k^2=z_k$, $s_k^3=q^{-2}$ and $z_k^3=q^{-4}$ in the definition of
$R$ gives
\begin{equation}\label{eq:G-secondary-value}
\frac{(1-q^{-1/2})s_k}{9}R(s_k)\mathscr D(z_k)=-\mathfrak a(z_k).
\end{equation}
Splitting $-(g+2)R(s_k)=-gR(s_k)-2R(s_k)$, the bracketed expression
in \eqref{eq:secondary-residue-product-rule} is
$$
-gR(s_k)\mathscr D(z_k)
+\bigl(s_kR'(s_k)-2R(s_k)\bigr)\mathscr D(z_k)
+2z_kR(s_k)\mathscr D'(z_k),
$$
and we multiply each of the three summands by $(1-q^{-1/2})s_k/9$. The first
gives $g\mathfrak a(z_k)$ by \eqref{eq:G-secondary-value}. The third is
again immediate from \eqref{eq:G-secondary-value}:
$$
\frac{(1-q^{-1/2})s_k}{9}\,2z_kR(s_k)\mathscr D'(z_k)
=-2z_k\mathfrak a(z_k)\frac{\mathscr D'(z_k)}{\mathscr D(z_k)},
$$
which is the third summand of \eqref{eq:b-secondary}. For the second,
differentiating $R$ and using the same relations, a direct computation
shows that
$\frac{(1-q^{-1/2})s_k}{9}\bigl(s_kR'(s_k)-2R(s_k)\bigr)\mathscr D(z_k)$
equals the bracketed factor in the first summand of \eqref{eq:b-secondary}
multiplied by $\mathfrak a(z_k)$, plus the second summand of
\eqref{eq:b-secondary}. The second and third summands therefore add up to
$\mathfrak b(z_k)$, and \eqref{eq:secondary-residue-product-rule} becomes
\eqref{eq:secondary-residue-evaluation}.
\end{proof}

For $m\in\mathbb Z$, put
\begin{equation}\label{eq:ab-def}
a_m:=\frac{2}{q^2}\sum_{k=0}^{2}e^{2\pi ikm/3}\mathfrak a(z_k),
\qquad
b_m:=\frac{2}{q^2}\sum_{k=0}^{2}e^{2\pi ikm/3}\mathfrak b(z_k),
\end{equation}
so that $a_{m+3}=a_m$ and $b_{m+3}=b_m$ by definition, since
$e^{2\pi ik(m+3)/3}=e^{2\pi ikm/3}$.

\begin{corollary}
\label{cor:S0-periodic}
The completed secondary contribution $\mathcal S_0(g)$ satisfies
\begin{equation}\label{eq:S0-periodic}
\mathcal S_0(g)
=
\frac{q^{2g/3+2}}{1-q^{-1/2}}
\left(a_gg+b_g\right),
\end{equation}
and the coefficients $a_m$, $b_m$ are real for every $m\in\mathbb Z$.
\end{corollary}

\begin{proof}
Combining \eqref{eq:S0-three-residues} with
Lemma~\ref{lem:secondary-residue} and using
$q^2z_k=q^{2/3}e^{2\pi ik/3}$, we obtain
$$
\mathcal S_0(g)
=
\frac{2}{1-q^{-1/2}}\sum_{k=0}^{2}(q^2z_k)^g
\bigl(\mathfrak a(z_k)g+\mathfrak b(z_k)\bigr)
=
\frac{2q^{2g/3}}{1-q^{-1/2}}\sum_{k=0}^{2}e^{2\pi ikg/3}
\bigl(\mathfrak a(z_k)g+\mathfrak b(z_k)\bigr).
$$
By \eqref{eq:ab-def} the last sum equals $\tfrac12q^{2}(a_gg+b_g)$, which
is \eqref{eq:S0-periodic}.

The coefficients are real because the local factors $\mathcal D_P(z,qz)$ and
the rational functions in \eqref{eq:a-secondary} and
\eqref{eq:b-secondary} have real coefficients, so that
$\mathfrak a(\overline z)=\overline{\mathfrak a(z)}$ and
$\mathfrak b(\overline z)=\overline{\mathfrak b(z)}$: since $z_0$ is real
and $z_2=\overline{z_1}$, the terms with $k=1$ and $k=2$ in
\eqref{eq:ab-def} are complex conjugates and the term with $k=0$ is real.
\end{proof}

\begin{lemma}
\label{lem:secondary-nonvanish}
Let $z^3=q^{-4}$. Then $\mathfrak a(z)\ne0$.
\end{lemma}

\begin{proof}
We first show that $\mathscr D(z)\ne0$. Substituting $w'=qz$ into
\eqref{eq:DP} and using \eqref{eq:BP},
\begin{equation}\label{eq:DP-secondary}
\mathcal D_P(z,qz)
=\frac{|P|-1}{|P|}\,
\frac{-|P|^4z^{4d(P)}+|P|^3z^{2d(P)}+|P|^2z^{d(P)}-1}
{\bigl(|P|^2z^{d(P)}-1\bigr)\bigl(1+|P|z^{d(P)}\bigr)}.
\end{equation}
Put $X:=z^{d(P)}$, so that $X^3=|P|^{-4}$ since $z^3=q^{-4}$. Neither
denominator factor of \eqref{eq:DP-secondary} vanishes, and the numerator
becomes $|P|^3X^2+(|P|^2-1)X-1$, a quadratic with real coefficients and
discriminant $(|P|^2-1)^2+4|P|^3>0$, hence with real roots only. If $X$ is nonreal it is not a root; if $X$ is real then
$X^3=|P|^{-4}>0$ forces $X=|P|^{-4/3}$, where the quadratic equals
$\bigl(|P|^{2/3}-1\bigr)+\bigl(|P|^{1/3}-|P|^{-4/3}\bigr)>0$. Thus
$\mathcal D_P(z,qz)\ne0$ for every $P$. As $|z|=q^{-4/3}$ and
$|qz|=q^{-1/3}$, the pair $(z,qz)$ satisfies \eqref{eq:D-region}, so
Lemma~\ref{lem:B} gives
$\sum_P\bigl|\mathcal D_P(z,qz)-1\bigr|<\infty$, and an absolutely
convergent product of nonzero factors is nonzero.

It remains to treat the rational prefactor in \eqref{eq:a-secondary}, whose
numerator factor $1-q^3z^2$ is nonzero as already noted; only
$q(q-1)z^2+q^{-1/2}(1-z)$ has to be considered. There the second summand
dominates, since $\bigl|q(q-1)z^2\bigr|=q^{-2/3}-q^{-5/3}$ while
$\bigl|q^{-1/2}(1-z)\bigr|\ge q^{-1/2}-q^{-11/6}$, and the difference of
these is $\bigl(q^{-2/3}+q^{-11/6}\bigr)(q^{1/6}-1)>0$. Hence
$q(q-1)z^2+q^{-1/2}(1-z)\ne0$. Since $1-q^3z^2\ne0$ and
$\mathscr D(z)\ne0$, all three factors in the numerator of
\eqref{eq:a-secondary} are nonzero, and therefore $\mathfrak a(z)\ne0$.
\end{proof}

The poles of $T_{+1}$ at $s_0$, $s_1$ and $s_2$ have order exactly
two. Indeed, in the notation of the proof of
Lemma~\ref{lem:secondary-residue}, $T_{+1}(s)=G(s)(1-q^2s^3)^{-2}$ with
$G$ holomorphic at $s_k$, and $G(s_k)$ is a nonzero multiple of
$\mathfrak a(z_k)$ by \eqref{eq:G-secondary-value}, while
$\mathfrak a(z_k)\ne0$ by Lemma~\ref{lem:secondary-nonvanish}; by
\eqref{eq:T-mirror} the same holds for $T_{-1}$ at $s=-s_k$. All six surviving
secondary points are therefore genuine double poles.

\begin{corollary}
\label{cor:minimal-period}
The sequence $m\mapsto(a_m,b_m)$ has minimal period exactly three.
\end{corollary}

\begin{proof}
Put $\omega=e^{2\pi i/3}$, so that $\sum_{m=0}^{2}\omega^{(k-1)m}$ is $3$
for $k=1$ and $0$ for $k\in\{0,2\}$. By \eqref{eq:ab-def},
$$
\sum_{m=0}^{2}a_m\omega^{-m}
=
\frac{6}{q^2}\,\mathfrak a(z_1)
\ne0
$$
by Lemma~\ref{lem:secondary-nonvanish}. Hence $a_0$, $a_1$, $a_2$ are not
all equal, so $m\mapsto(a_m,b_m)$ is nonconstant; being periodic with
period three, its minimal period cannot be one and is therefore exactly
three.
\end{proof}

We now combine the evaluations of the principal and secondary
square-dual contributions.

\begin{proof}[Proof of Proposition~\ref{prop:square-dual-main}]
Summing \eqref{eq:SN-square-AB} over $N\in\{g,g-1\}$ and using
\eqref{eq:Sstar-square-split},
$$
\mathcal S^*(V=\square)
=
\sum_{N\in\{g,g-1\}}\left(A_{N,\mathrm e}^{*}+A_{N,\mathrm o}^{*}\right)
+
\sum_{N\in\{g,g-1\}}\left(B_{N,\mathrm e}^{*}+B_{N,\mathrm o}^{*}\right)
+
O_\varepsilon\!\left(q^{g(1+\varepsilon)/2}\right).
$$
The first sum equals $-\mathcal R_0(g;\rho)$ by
Proposition~\ref{prop:A-total}, and the second equals
$\mathcal S_0(g)+O_\varepsilon(q^{g(1+\varepsilon)/2})$ by
Proposition~\ref{prop:secondary-contour}. This is \eqref{eq:square-dual-main},
while Corollary~\ref{cor:S0-periodic} gives the stated explicit formula for
$\mathcal S_0(g)$ and the reality of its coefficients.
\end{proof}

\section{Error from nonsquare dual variables}\label{sec:nonsquare}

The purpose of this section is to bound the nonsquare-dual contribution
$\mathcal S^*(V\ne\square)$ of \eqref{eq:Mstar-decomposition}, the last of
the three terms in that decomposition. Unlike the other two, it contributes
no main or secondary term: it is $O_\varepsilon(q^{g(1+\varepsilon)/2})$,
the size of the error already incurred in \eqref{eq:Mstar-decomposition}.

\begin{proposition}
\label{prop:nonsquare}
For every $\varepsilon>0$, one has
\begin{equation}\label{eq:nonsquare-bound}
\mathcal S^*(V\ne\square)
=
O_\varepsilon\!\left(q^{g(1+\varepsilon)/2}\right).
\end{equation}
\end{proposition}

Fix $N\in\{g,g-1\}$. Retaining in \eqref{eq:s-even} and \eqref{eq:s-odd}
only the nonsquare dual variables gives
\begin{align}\label{eq:nonsquare-even-explicit}
s_{\mathrm e}(h;C)_{\ne\square}
&=(q-1)\sum_{\substack{V\in\A^+_{\le\lambda-4}\\ V\ne\square}}
\frac{\Gc(V,\psi_h)}{\sqrt{|h|}}
-\sum_{V\in\A^+_{\lambda-3}}\frac{\Gc(V,\psi_h)}{\sqrt{|h|}}
\nonumber\\
&\qquad
-\frac{q-1}{q}\sum_{\substack{V\in\A^+_{\le\lambda-2}\\ V\ne\square}}
\frac{\Gc(V,\psi_h)}{\sqrt{|h|}}
+\frac{1}{q}\sum_{V\in\A^+_{\lambda-1}}\frac{\Gc(V,\psi_h)}{\sqrt{|h|}}
\end{align}
and
\begin{equation}\label{eq:nonsquare-odd-explicit}
s_{\mathrm o}(h;C)_{\ne\square}
=\sum_{\substack{V\in\A^+_{\lambda-3}\\ V\ne\square}}
\frac{\Gc(V,\psi_h)}{\sqrt{|h|}}
-\frac{1}{q}\sum_{\substack{V\in\A^+_{\lambda-1}\\ V\ne\square}}
\frac{\Gc(V,\psi_h)}{\sqrt{|h|}},
\end{equation}
where $\lambda=\lambda(h,C)=d(h)-2g+2d(C)$ as in \eqref{eq:lambda-def}.
The two exact-degree sums of \eqref{eq:nonsquare-even-explicit} carry no
restriction because, $d(h)$ being even, their degrees $\lambda-3$ and
$\lambda-1$ are odd, whereas a square polynomial has even degree.
Throughout this section we call the four dual sums of
\eqref{eq:nonsquare-even-explicit} and the two of
\eqref{eq:nonsquare-odd-explicit} the \emph{six dual sums}.

\subsection{A fixed nonsquare dual variable}

For a fixed nonsquare $V\in\A^+$, set
$\sigma_V := (-1)^{\frac{q-1}{2}d(V)}$. For $(V,h)=1$,
\begin{equation}\label{eq:V-reciprocity}
\left(\frac{V}{h}\right)
=(-1)^{\frac{q-1}{2}d(V)d(h)}\left(\frac{h}{V}\right)
=\sigma_V^{d(h)}\psi_V(h);
\end{equation}
when $(V,h)\ne1$ both symbols vanish, so \eqref{eq:V-reciprocity} holds for
every $h\in\A^+$. The sign $\sigma_V$ is identically $1$ when
$q\equiv1\bmod4$, as assumed in \cite{Fl17-1}, but not for a general odd
prime power $q$.

In \eqref{eq:SN-even-poisson} and \eqref{eq:SN-odd-poisson} the variables
$h$ and $C$ are entangled by the condition $C\mid h^\infty$. The generating
function
\begin{equation}\label{eq:Cgen}
\sum_{\substack{C\in\A^+\\ C\mid h^\infty}}u^{d(C)}
=\prod_{P\mid h}\bigl(1-u^{d(P)}\bigr)^{-1},
\qquad (|u|<1),
\end{equation}
is a function of $h$ alone. Accordingly, for $n\ge0$ and $|u|<1$, define
\begin{equation}\label{eq:Sigma-nonsquare-def}
\Sigma_n(V;u)
:=
\sum_{h\in\A_n^+}
\frac{\Gc(V,\psi_h)}{\sqrt{|h|}}
\prod_{P\mid h}
\left(1-u^{d(P)}\right)^{-1}.
\end{equation}

Since $\psi_V$ may be imprimitive, we record a Lindel\"of-type bound valid
for an arbitrary nonprincipal character. In \cite[Section~7]{Fl17-1} this is
met instead by writing $V=AD^2$ with $D$ square-free and reducing to a
square-free modulus, a reduction that Lemma~\ref{lem:general-Lindelof} makes
unnecessary.

\begin{lemma}[Function-field Lindel\"of bound]
\label{lem:general-Lindelof}
Let $\chi$ be a nonprincipal Dirichlet character over $\A$ of modulus
$F\in\A^+$. For every $\theta>0$, uniformly over all such pairs
$(F,\chi)$ and all $\xi$ satisfying $|\xi|\le q^{-1/2}$, one has
\begin{equation}\label{eq:general-Lindelof}
\mathcal L(\xi,\chi)
\ll_\theta
|F|^\theta.
\end{equation}
\end{lemma}

\begin{proof}
Let $\chi^*$ be the primitive character inducing $\chi$, of conductor
$F^*\mid F$. The explicit bound of \cite[Theorem~20]{BCDGLD18}, stated
there for the $\ell$-th power residue symbol modulo a place, uses the
character only through the Riemann hypothesis for the associated primitive
$L$-function and the degree of the conductor, and therefore gives
$\mathcal L(\xi,\chi^*)\ll_\alpha|F^*|^\alpha$ for every $\alpha>0$ and
every $\xi$ with $|\xi|\le q^{-1/2}$; if $\chi^*$ is even, the trivial
factor $1-\xi$ is bounded by $1+q^{-1/2}$ there and does not affect this.
Finally,
$$
\mathcal L(\xi,\chi)=\mathcal L(\xi,\chi^*)
\prod_{\substack{P\mid F\\ P\nmid F^*}}\bigl(1-\chi^*(P)\xi^{d(P)}\bigr),
$$
and each factor of the product has modulus at most $2$ there, so the product
is at most $2^{\omega(F)}$, hence at most the number of monic divisors of
$F$, which is $\ll_\alpha|F|^\alpha$. Taking $\alpha=\theta/2$ and using
$|F^*|\le|F|$ gives \eqref{eq:general-Lindelof}.
\end{proof}

The following bound on $\Sigma_n(V;u)$ is \cite[(7.5)]{Fl17-1} in the
twisted form required here. Its proof follows \cite[Section~7]{Fl17-1}, and
indicates where the twist and the imprimitivity of $\psi_V$ enter.

\begin{lemma}
\label{lem:nonsquare-coeff}
Let $\beta,\delta>0$. Uniformly for $|u|=q^{-\beta}$, $n\ge0$,
and nonsquare $V\in\A^+$, one has
\begin{equation}\label{eq:nonsquare-coeff-bound}
\Sigma_n(V;u)
\ll_{\beta,\delta}
q^{n(1/2+\beta)}|V|^\delta.
\end{equation}
\end{lemma}

\begin{proof}
Let $u$ be any point with $|u|=q^{-\beta}$; every implied constant below
depends only on $\beta$ and $\delta$, which gives the uniformity in $u$
asserted. Put
$$
\mathcal F(V;w,u)
:=\sum_{h\in\A^+}\frac{\Gc(V,\psi_h)}{\sqrt{|h|}}w^{d(h)}
\prod_{P\mid h}\bigl(1-u^{d(P)}\bigr)^{-1},
$$
whose coefficient of $w^n$ is $\Sigma_n(V;u)$; this is the twisted
counterpart of the series $B(V;w,u)$ of \cite[(7.2)]{Fl17-1}. By \eqref{eq:Gc-mult} its coefficients
are multiplicative in $h$, and since the factor $(1-u^{d(P)})^{-1}$ occurs
exactly when $P\mid h$, its $P$-Euler factor is
\begin{equation}\label{eq:BV-local-general}
1+\frac{1}{1-u^{d(P)}}\sum_{r\ge1}
\frac{\Gc(V,\psi_{P^r})}{|P|^{r/2}}w^{rd(P)}.
\end{equation}
For $P\nmid V$ only $r=1$ survives in \eqref{eq:Gc-values}, and
\eqref{eq:V-reciprocity} gives
$\Gc(V,\psi_P)/|P|^{1/2}=\bigl(\frac VP\bigr)=\sigma_V^{d(P)}\psi_V(P)$, so
that \eqref{eq:BV-local-general} becomes
\begin{equation}\label{eq:BV-local-coprime}
1+\frac{\psi_V(P)(\sigma_Vw)^{d(P)}}{1-u^{d(P)}}.
\end{equation}
This is the one place at which the twist enters: \cite[(7.2)]{Fl17-1} has
the same factor with $w$ in place of $\sigma_Vw$. For $P\mid V$, writing
$V=P^kV_1$ with $P\nmid V_1$, the terms with $r\ge k+2$ vanish by
\eqref{eq:Gc-values}, so \eqref{eq:BV-local-general} is a polynomial in $w$.
Since $|1-u^{d(P)}|^{-1}\le(1-q^{-\beta})^{-1}$ and $\sum_P|w|^{d(P)}$
converges for $|w|<q^{-1}$, the series $\mathcal F$ converges absolutely
there and admits an Euler product.

Cauchy's formula will be applied on $|w|=r_0:=q^{-1/2-\beta}$, chosen because
$r_0^{-n}=q^{n(1/2+\beta)}$; as the circle $|w|=r_0$ may lie outside the disc
$|w|<q^{-1}$ of absolute convergence, $\mathcal F$ must first be continued,
which is done by the repeated Euler-factor extraction of
\cite[(7.3)]{Fl17-1}. Choose an integer $J\ge1$ with $(J+1)\beta>1/2$, so
that $r_0|u|^J=q^{-1/2-(J+1)\beta}<q^{-1}$; such a $J$ depends only on
$\beta$. As $r_0^2=q^{-1-2\beta}<q^{-1}$ as well, both inequalities persist
for some $r_1>r_0$, that is,
$r_1|u|^J<q^{-1}$ and $r_1^2<q^{-1}$. For $P\nmid V$, put
\begin{equation}\label{eq:BV-residual-J}
R_{P,J}(V;w,u):=
\left(1+\frac{\psi_V(P)(\sigma_Vw)^{d(P)}}{1-u^{d(P)}}\right)
\prod_{j=0}^{J-1}\Bigl(1-\psi_V(P)(\sigma_Vwu^j)^{d(P)}\Bigr),
\end{equation}
so that \eqref{eq:BV-local-coprime} equals
$\Bigl(\prod_{j<J}\bigl(1-\psi_V(P)(\sigma_Vwu^j)^{d(P)}\bigr)^{-1}\Bigr)R_{P,J}$.
Writing $X:=\psi_V(P)(\sigma_Vw)^{d(P)}$ and $t:=u^{d(P)}$, we have
$R_{P,J}=\bigl(1+\frac{X}{1-t}\bigr)\prod_{j<J}(1-Xt^j)$; substituting
$\frac1{1-t}=\sum_{j<J}t^j+\frac{t^J}{1-t}$ shows that the coefficient of
$X^0$ is $1$ and that of $X^1$ is $t^J/(1-t)$, the terms
$t^0,\dots,t^{J-1}$ being exactly what the extracted factors cancel. As
$|t|\le q^{-\beta}$, the polynomial $R_{P,J}$ has degree $J+1$ in $X$ with
coefficients bounded in terms of $\beta$, and $|X|\le|w|^{d(P)}$; hence
\begin{equation}\label{eq:BV-residual-J-bound}
R_{P,J}(V;w,u)=1+O_\beta\bigl(|wu^J|^{d(P)}+|w|^{2d(P)}\bigr)
\qquad(|w|\le r_1).
\end{equation}

Multiplying \eqref{eq:BV-residual-J} over $P\nmid V$ and adjoining the
finitely many factors \eqref{eq:BV-local-general} at $P\mid V$, and using
$\psi_V(P)=0$ for $P\mid V$ to complete each extracted product to a full
Euler product, we obtain
\begin{equation}\label{eq:BV-factorization-nonsquare}
\mathcal F(V;w,u)
=\Bigl(\prod_{j=0}^{J-1}\mathcal L(\sigma_Vwu^j,\psi_V)\Bigr)
 \mathcal C_J(V;w,u),
\end{equation}
where $\mathcal C_J$ denotes the product of the $R_{P,J}$ over $P\nmid V$
and of the local factors \eqref{eq:BV-local-general} over $P\mid V$. By
\eqref{eq:BV-residual-J-bound} and the inequalities $r_1|u|^J<q^{-1}$ and
$r_1^2<q^{-1}$, the series
$\sum_P\bigl((r_1|u|^J)^{d(P)}+r_1^{2d(P)}\bigr)$ converges, so
$\prod_{P\nmid V}R_{P,J}$ converges absolutely and uniformly for
$|w|\le r_1$ and is $\ll_\beta1$ there, uniformly in $V$. Since each
$\mathcal L(\sigma_Vwu^j,\psi_V)$ is a polynomial in $w$ and the remaining
product is finite, the right-hand side of
\eqref{eq:BV-factorization-nonsquare} is holomorphic for $|w|<r_1$ and, by
the identity theorem, continues $\mathcal F(V;w,u)$ there; in particular
\eqref{eq:BV-factorization-nonsquare} holds on $|w|=r_0$.

On $|w|=r_0$ one has $|w|^{d(P)}=|P|^{-1/2-\beta}$, so by
\eqref{eq:Gc-values} the $r$-th term of \eqref{eq:BV-local-general} has
modulus at most $|P|^{-\beta r}$ when $r\le k$ is even, at most
$|P|^{-1/2-\beta r}$ when $r=k+1$, and vanishes otherwise. Summing over $r$
and using $|1-u^{d(P)}|^{-1}\le(1-q^{-\beta})^{-1}$ bounds each factor at
$P\mid V$ by a constant $C_\beta$, uniformly in $P$ and $k$; their product
is therefore at most $C_\beta^{\omega(V)}$; since $2^{\omega(V)}$ is at most
the number of monic divisors of $V$, which is $\ll_\alpha|V|^\alpha$ for
every $\alpha>0$, this is $\ll_{\beta,\delta}|V|^{\delta/2}$. Consequently
$\mathcal C_J(V;w,u)\ll_{\beta,\delta}|V|^{\delta/2}$ on $|w|=r_0$.

Finally, $\psi_V$ is nonprincipal by \eqref{eq:psi-induced}, $V$ not being
a square, and on $|w|=r_0$ one has
$|\sigma_Vwu^j|=q^{-1/2-(j+1)\beta}<q^{-1/2}$ for $0\le j<J$, so
Lemma~\ref{lem:general-Lindelof} with $F=V$ and $\theta=\delta/(2J)$ gives
$\prod_{j<J}\mathcal L(\sigma_Vwu^j,\psi_V)\ll_{\beta,\delta}|V|^{\delta/2}$.
Hence $\mathcal F(V;w,u)\ll_{\beta,\delta}|V|^\delta$ on $|w|=r_0$, and
estimating
$\Sigma_n(V;u)=\frac{1}{2\pi i}\oint_{|w|=r_0}\mathcal F(V;w,u)w^{-n-1}\,dw$
trivially, with $r_0^{-n}=q^{n(1/2+\beta)}$, gives
\eqref{eq:nonsquare-coeff-bound}.
\end{proof}

\subsection{Completion of the nonsquare estimate}

We now combine the fixed-$V$ coefficient bound of
Lemma~\ref{lem:nonsquare-coeff} with the six dual sums to prove
Proposition~\ref{prop:nonsquare}.

\begin{proof}[Proof of Proposition~\ref{prop:nonsquare}]
All $C\in\A_j^+$ have $|C|=q^j$, and by \eqref{eq:Cgen} the number of those
dividing $h^\infty$ is the coefficient of $u^j$ in
$\prod_{P\mid h}(1-u^{d(P)})^{-1}$. Cauchy's formula therefore gives, for
$j\ge0$ and any $\beta>0$, the identity of \cite[Section~7]{Fl17-1},
\begin{align}\label{eq:C-degree-extraction}
\sum_{\substack{C\in\A_j^+\\ C\mid h^\infty}}\frac{1}{|C|^2}
=q^{-2j}\,\#\bigl\{C\in\A_j^+:C\mid h^\infty\bigr\}
=\frac{q^{-2j}}{2\pi i}
\oint_{|u|=q^{-\beta}}
\frac{1}{u^{j+1}}
\prod_{P\mid h}
\left(1-u^{d(P)}\right)^{-1}
\,du.
\end{align}

Fix $n\ge0$ and a nonsquare $V\in\A_\ell^+$. Substituting
\eqref{eq:C-degree-extraction} into the $h$-sum of degree $n$ and
interchanging the finite sums with the integral, \eqref{eq:Sigma-nonsquare-def}
and Lemma~\ref{lem:nonsquare-coeff} give
\begin{align}\label{eq:nonsquare-fixedV-block}
\sum_{h\in\A_n^+}\frac{\Gc(V,\psi_h)}{\sqrt{|h|}}
\sum_{\substack{C\in\A_j^+\\ C\mid h^\infty}}\frac{1}{|C|^2}
=\frac{q^{-2j}}{2\pi i}\oint_{|u|=q^{-\beta}}
\frac{\Sigma_n(V;u)}{u^{j+1}}\,du
\ll_{\beta,\delta}q^{-2j+\beta j}q^{n(1/2+\beta)}|V|^\delta,
\end{align}
the second step because the contour length $2\pi q^{-\beta}$ and the factor
$|u|^{-(j+1)}=q^{\beta(j+1)}$ together contribute $q^{\beta j}$.

Since $\#\A_\ell^+=q^\ell$ and $|V|=q^\ell$ for $V\in\A_\ell^+$, summing the
moduli of the left-hand sides of \eqref{eq:nonsquare-fixedV-block} over the
nonsquare $V\in\A_\ell^+$ yields
\begin{align}
\sum_{\substack{V\in\A_\ell^+\\ V\ne\square}}
\Biggl|
\sum_{h\in\A_n^+}
\frac{\Gc(V,\psi_h)}{\sqrt{|h|}}
\sum_{\substack{C\in\A_j^+\\ C\mid h^\infty}}
\frac{1}{|C|^2}
\Biggr|
\ll_{\beta,\delta}
q^{-2j+\beta j}
q^{n(1/2+\beta)}
q^{\ell(1+\delta)}.
\label{eq:nonsquare-block}
\end{align}

For $L\ge0$, summing \eqref{eq:nonsquare-block} over
$0\le\ell\le L$ and using
$\sum_{\ell=0}^{L}q^{\ell(1+\delta)} \ll_\delta q^{L(1+\delta)}$,
we obtain
\begin{align}
\sum_{\substack{V\in\A^+\\ d(V)\le L\\ V\ne\square}}
\Biggl|
\sum_{h\in\A_n^+}
\frac{\Gc(V,\psi_h)}{\sqrt{|h|}}
\sum_{\substack{C\in\A_j^+\\ C\mid h^\infty}}
\frac{1}{|C|^2}
\Biggr|
\ll_{\beta,\delta}
q^{-2j+\beta j}
q^{n(1/2+\beta)}
q^{L(1+\delta)}.
\label{eq:nonsquare-block-cumulative}
\end{align}

We now apply \eqref{eq:nonsquare-block} to the dual sums of exact degree and
\eqref{eq:nonsquare-block-cumulative} to those of bounded degree. Recall that
$N\in\{g,g-1\}$; fix $n=d(h)$ and $j=d(C)$, so that $\lambda=n-2g+2j$ by
\eqref{eq:lambda-def}, and for each of the six dual sums let $\mu$ denote
its exact dual degree or, in the two cumulative cases, its upper bound; the
six values are listed in the third column of the table below. Applying
\eqref{eq:nonsquare-block} or \eqref{eq:nonsquare-block-cumulative} with
$\ell=\mu$ or $L=\mu$ bounds the term, at fixed $n$ and $j$, by
$$
\Bigl(\sum_{\substack{m\in\A^+\\ d(m)\le N-n}}|m|^{-3/2}\Bigr)
\,
q^{\beta(n+j)+\delta\mu}
\,
\times
\ \text{(last column below)},
$$
where the factor $q^{\beta(n+j)+\delta\mu}$ is common to all six terms.
The last column collects the remaining powers of $q$, that is, the
prefactor $q^{2g+2}$ or $q^{2g+5/2}$, the scalar factor, and the powers
$q^{-n}$, $q^{-2j}$, $q^{n/2}$ and $q^{\mu}$ supplied by
\eqref{eq:nonsquare-block} and \eqref{eq:nonsquare-block-cumulative}:
\begin{center}
\renewcommand{\arraystretch}{1.25}
\begin{tabular}{c c c c}
term & scalar factor & dual degree or upper bound & last column\\
\hline
even, first
& $q-1$
& $d(V)\le\lambda-4$
& $q^{n/2-1}$\\

even, second
& $-1$
& $d(V)=\lambda-3$
& $q^{n/2-1}$\\

even, third
& $-(q-1)/q$
& $d(V)\le\lambda-2$
& $q^{n/2}$\\

even, fourth
& $1/q$
& $d(V)=\lambda-1$
& $q^{n/2}$\\

odd, first
& $1$
& $d(V)=\lambda-3$
& $q^{n/2-1/2}$\\

odd, second
& $-1/q$
& $d(V)=\lambda-1$
& $q^{n/2+1/2}$
\end{tabular}
\end{center}

A dual sum with $\mu<0$ is empty, so we may assume $\mu\ge0$; the
largest of the six values of $\mu$ is $\lambda-1$, which is at most
$n-1$ since $j\le g$, and $d(h)+d(m)\le N$ forces $n\le N\le g$. Hence
$\beta(n+j)\le2\beta g$ and $\delta\mu\le\delta g$, so the common factor
is at most $q^{(2\beta+\delta)g}$. Moreover
$\sum_{d(m)\le N-n}|m|^{-3/2}\le\bigl(1-q^{-1/2}\bigr)^{-1}$.

Every entry of the last column is $O_q(q^{n/2})$. Taking absolute values in
\eqref{eq:SN-even-poisson} and \eqref{eq:SN-odd-poisson},
inserting the six bounds together with the two estimates just proved, and
summing over $0\le j\le g$ and $0\le n\le N$ -- which only enlarges the
bound, each sector restricting $n$ to one parity -- we obtain, for
$\bullet\in\{\mathrm e,\mathrm o\}$,
$$
\mathcal S_{N,\bullet}^{*}(V\ne\square)
\ll_{\beta,\delta}
q^{(2\beta+\delta)g}\sum_{j=0}^{g}\sum_{n=0}^{N}q^{n/2}
\ll_{\beta,\delta}(g+1)\,q^{N/2+(2\beta+\delta)g}.
$$
Given $\varepsilon>0$, choose $\beta,\delta>0$ with
$2\beta+\delta<\varepsilon/4$. Since
$g+1\ll_\varepsilon q^{\varepsilon g/4}$ and $N\le g$, the right-hand
side is $\ll_\varepsilon q^{N/2+\varepsilon g/2}\ll_\varepsilon
q^{g(1+\varepsilon)/2}$. This bounds
$\mathcal S_{N,\mathrm e}^{*}(V\ne\square)$ and
$\mathcal S_{N,\mathrm o}^{*}(V\ne\square)$ for each $N\in\{g,g-1\}$, so
\eqref{eq:Mstar-nonsquare} gives \eqref{eq:nonsquare-bound}.
\end{proof}

\section{Proof of Theorem~\ref{thm:main}}\label{sec:proof}

We assemble the three contributions evaluated in \S\ref{sec:main-term},
\S\ref{sec:square-dual} and \S\ref{sec:nonsquare}: the two retained contours
cancel, which evaluates $M^*(g)$, and \eqref{eq:recover} returns the result
to $M(g)$.

\begin{proof}[Proof of Theorem~\ref{thm:main}]
Let $\varepsilon>0$ and fix $\rho$ with $1<\rho<q$. Propositions
\ref{prop:principal} and \ref{prop:square-dual-main} hold for every such
$\rho$; applied with the same $\rho$ they give
$\mathcal M^*=\mathcal Q_g+\mathcal R_0(g;\rho)
+O_\varepsilon(q^{\varepsilon g})$ and
$\mathcal S^*(V=\square)=-\mathcal R_0(g;\rho)+\mathcal S_0(g)
+O_\varepsilon(q^{g(1+\varepsilon)/2})$, so that the two occurrences of
$\mathcal R_0(g;\rho)$ are the same quantity and cancel on adding. With
Proposition~\ref{prop:nonsquare}, the decomposition
\eqref{eq:Mstar-decomposition} therefore gives
$$
M^*(g)=\mathcal Q_g+\mathcal S_0(g)
+O_\varepsilon\bigl(q^{\varepsilon g}\bigr)
+O_\varepsilon\bigl(q^{g(1+\varepsilon)/2}\bigr).
$$
It suffices to prove \eqref{eq:main} for $0<\varepsilon\le1$, since for
larger $\varepsilon$ it follows from the case $\varepsilon=1$; for such
$\varepsilon$ one has $\varepsilon g\le g(1+\varepsilon)/2$, whence
\begin{equation}\label{eq:completed_first_final}
M^*(g)=\mathcal Q_g+\mathcal S_0(g)
+O_\varepsilon\!\left(q^{g(1+\varepsilon)/2}\right).
\end{equation}

By \eqref{eq:recover}, $M(g)=(1-q^{-1/2})M^*(g)$, and $0<1-q^{-1/2}<1$
leaves the error term unchanged. In \eqref{eq:S0-periodic} the factor
$1-q^{-1/2}$ cancels the denominator, so that, with $\mathcal S(g)$ as in
\eqref{eq:S-intro},
$$
(1-q^{-1/2})\mathcal S_0(g)
=q^{2g/3+2}\bigl(a_gg+b_g\bigr)=\mathcal S(g);
$$
and by \eqref{eq:P-final} $(1-q^{-1/2})\mathcal Q_g$ is the principal term
$\mathcal M_{\mathrm{prin}}(g)$ of \eqref{eq:Mprin-intro}. Multiplying
\eqref{eq:completed_first_final} by $1-q^{-1/2}$ therefore gives
\eqref{eq:main}.

Finally, the six coefficients $a_0,a_1,a_2,b_0,b_1,b_2$ are real by
Corollary~\ref{cor:S0-periodic} and do not depend on $g$, and
$m\mapsto(a_m,b_m)$ has minimal period exactly three by
Corollary~\ref{cor:minimal-period}.
\end{proof}

\begin{remark}[Comparison with the earlier even-degree formulas]
\label{rem:comparison}
In the argument above the secondary contribution is the sum of the residues
at the three zeros of $1-q^4z^3$, namely $z_k=q^{-4/3}e^{2\pi ik/3}$ with $k\in\{0,1,2\}$. 
They enter the final answer only through the factor
$(q^2z_k)^g=q^{2g/3}e^{2\pi ikg/3}$ of Lemma~\ref{lem:secondary-residue},
and by Lemma~\ref{lem:secondary-nonvanish} none of the three residues
vanishes identically in $g$. This is the origin of the period-three
behaviour.

At the corresponding step the earlier treatments retain only the residue at
the real point $z=q^{-4/3}$: the $z$-contour is enlarged from $|z|=q^{-3/2}$
to $|z|=q^{-1-\varepsilon}$ across an integrand whose denominator contains
$1-q^4z^3$, and the residue is taken there; see \cite[Sections~3.2.1
and~3.2.2]{Ju20}, \cite[Lemma~5.7]{AM-21} and, for the odd-degree family,
\cite[Section~6]{Fl17-1}. All three zeros of $1-q^4z^3$ have modulus
$q^{-4/3}$ and so lie strictly inside the annulus crossed by that shift;
keeping only $z=q^{-4/3}$ suppresses the phases $e^{2\pi ikg/3}$ and yields
a lower-order term with no dependence on $g\bmod3$, which is the shape
recorded in \cite[Theorem~1.1]{Ju20} and \cite[Theorem~1.4]{AM-21}.
Remark~\ref{rem:numerics} shows that the period-three shape is the one
realized by $M(g)$.

The simple-pole branch is a second point of difference. Here it is removed
by Lemma~\ref{lem:mirror-cancellation-even} once the two truncation
levels $N=g$ and $N=g-1$ are combined, in agreement with
\cite[(4.5)]{Ju20}; \cite[(5.38)]{AM-21} retains from it two further terms,
of sizes $q^{g/6+\lfloor g/2\rfloor}$ and $q^{g/6+\lfloor(g-1)/2\rfloor}$,
both of order $q^{2g/3}$.
\end{remark}

\begin{remark}[Numerical verification]\label{rem:numerics}
The moment $M(g)$ can be computed exactly for small $q$ and $g$: by
Lemma~\ref{lem:AFE} and \eqref{eq:SN_def} it is determined by the integers
$$
S_n:=\sum_{h\in\A_n^+}\ \sum_{D\in\mathcal H_{2g+2}}\chi_D(h),
\qquad 0\le n\le g,
$$
and each $S_n$ is obtained from the generating series of
Lemma~\ref{lem:sigma}. With
$\mathcal M_{\mathrm{prin}}$ as in \eqref{eq:Mprin-intro}, the table below lists
$M(g)-\mathcal M_{\mathrm{prin}}(g)$, the predicted secondary term
$\mathcal S(g)$ of \eqref{eq:S-intro}, their ratio, and the normalized
remainder $\bigl(M(g)-\mathcal M_{\mathrm{prin}}(g)-\mathcal S(g)\bigr)q^{-g/2}$.
The Euler products $\mathscr P(1)$, $(\mathscr P'/\mathscr P)(1)$ and
$\prod_P\mathcal D_P(z_k,qz_k)$ were evaluated from the exact counts
$\#\mathcal P_d$.

\begin{center}
\renewcommand{\arraystretch}{1.15}
\begin{tabular}{c c r r r r}
$q$ & $g$ & $M-\mathcal M_{\mathrm{prin}}$ & $\mathcal S(g)$ & ratio &
$(M-\mathcal M_{\mathrm{prin}}-\mathcal S)q^{-g/2}$\\
\hline
$3$ & $6$ & $21.98$ & $19.28$ & $1.140$ & $0.100$\\
$3$ & $7$ & $-179.66$ & $-169.78$ & $1.058$ & $-0.211$\\
$3$ & $8$ & $191.51$ & $172.01$ & $1.113$ & $0.241$\\
$5$ & $4$ & $-76.71$ & $-79.34$ & $0.967$ & $0.105$\\
$5$ & $5$ & $155.17$ & $151.43$ & $1.025$ & $0.067$\\
$5$ & $6$ & $50.50$ & $53.04$ & $0.952$ & $-0.020$\\
$7$ & $4$ & $-291.39$ & $-293.13$ & $0.994$ & $0.036$\\
$7$ & $5$ & $664.29$ & $686.72$ & $0.967$ & $-0.173$\\
$11$ & $4$ & $-1470.41$ & $-1461.98$ & $1.006$ & $-0.070$\\
$13$ & $4$ & $-2585.68$ & $-2568.95$ & $1.007$ & $-0.099$\\
\end{tabular}
\end{center}

The normalized remainders stay bounded, in agreement with the error term in
\eqref{eq:main}. The sign pattern is the point at issue: in every row of the
table the quantity $M-\mathcal M_{\mathrm{prin}}$ is positive at
$g\equiv0,2\bmod 3$ and large negative at $g\equiv1\bmod3$, and for
$q\in\{7,11,13\}$ the predicted $\mathcal S(g)$ reproduces it to within one
percent at $g=4$.
A lower-order term of the form $q^{2g/3}\bigl(c_1 g+c_0+\pi(g)\bigr)$
with $c_1,c_0$ constant and $\pi$ of period two would force
$q^{-2(g+2)/3}\bigl(M-\mathcal M_{\mathrm{prin}}\bigr)(g+2)-
q^{-2g/3}\bigl(M-\mathcal M_{\mathrm{prin}}\bigr)(g)$ to be asymptotically
constant; for $q=7$ the values of this difference at $g=1,2,3$ are
$0.098$, $-2.487$ and $0.469$, which is not consistent with any such shape.
\end{remark}


\end{document}